\documentclass[reqno]{amsart}
\usepackage{geometry}
\usepackage{csquotes}
\usepackage{graphicx} % Required for inserting images
\usepackage{todonotes}
\usepackage{mathtools}
\usepackage{esint}
\usepackage[bookmarks = true,colorlinks = true]{hyperref}

\allowdisplaybreaks

\title[The Aronson-Bénilan estimate for a Lagrangian discretization of the PME]{The Aronson-Bénilan estimate for a Lagrangian particle discretization of the Porous Medium Equation}
\author{Marco Di Francesco}
\address{Marco Di Francesco - Dipartimento di Ingegneria e Scienze dell'Informazione e Matematica (DISIM),
Università degli Studi dell'Aquila,
Via Vetoio 1, Coppito,
67100 L'Aquila, Italy.}
\email{marco.difrancesco@univaq.it}
\author{Daniel Matthes}
\address{Daniel Matthes - Department of Mathematics,
TUM School of Computation, Information and Technology,
Technische Universität München, Boltzmannstr. 3
85748 Garching b. München, Germany.}
\email{matthes@ma.tum.de}

\date{January 2026}

\begin{document}

\newtheorem{proposition}{Proposition}
\newtheorem{lemma}{Lemma}
\newtheorem{theorem}{Theorem}
\newtheorem{corollary}{Corollary}

\newcommand{\R}{{\mathbb R}}
\newcommand{\Rp}{{\mathbb R}_{\geq0}}
\newcommand{\Rnn}{{\mathbb R}_{\ge0}}
\newcommand{\Z}{{\mathbb Z}}
\newcommand{\intR}{\int_{\mathbb R}}
\newcommand{\dd}{\,\mathrm{d}}
\newcommand{\dn}{\mathrm{d}}
\newcommand{\dff}{\operatorname{D}}
\newcommand{\tvn}[1]{\left\|#1\right\|_\text{TV}}
\newcommand{\fnc}{{\mathcal F}}
\newcommand{\rec}{\mathfrak{R}}

\newcommand{\rmap}{\mathcal{R}}
\newcommand{\recon}{\mathfrak{Rec}}
\newcommand{\sam}{\mathbf{S}}

\begin{abstract}
  We consider a nearest neighbor, Lagrangian particle discretization of the one dimensional porous medium equation. We prove that the particle model satisfies a discrete analog of the celebrated Aronson-Bénilan estimate, which we use to prove a growth estimate for the evolution of the support and an $L^\infty$ decay estimate which are both known to hold in the continuum. These estimates are uniform with respect to the number of particles. We also prove convergence of the scheme towards the solution to the porous medium equation in the full generality of $L^1$ initial data.
\end{abstract}



\maketitle

%\tableofcontents

\section{Introduction}

The porous medium equation
\begin{equation}\label{eq:main_PME}
    \partial_t \rho = \Delta \rho^m\,,\qquad m>1
\end{equation}
is a classical model describing the filtration of an isentropic gas through a porous medium \cite{barenblatt}, or the evolution of temperature in radiating media \cite{zeldovich}. More recently, it has been used in population dynamics as well, to describe phenomena of nonlinear diffusion or local repulsion between interacting agents e.g. in swarms \cite{mogilner}. In \eqref{eq:main_PME}, $\rho=\rho(t,x)$ represents the (nonnegative) density of a continuum medium, depending on a space variable $x\in \R^d$ and time $t\geq 0$.
Unlike its \enquote{linear closest relative}, the heat equation, which features solutions supported on the whole space $\R^d$ regardless of the initial datum, the porous medium equation features a finite-speed growth of the support, a property that makes it more preferable to its linear counterpart in many applications. The mathematical theory of the porous medium equation \eqref{eq:main_PME} has seen numerous and significant contributions since the 1980s. Far from claiming to be exhaustive, we cite here \cite{AB,AC,barenblatt,CF,dibenedetto,V1,V2,pierre_nonlinear_analysis_1982,CT,otto} and refer to the monographs \cite{Vazquez} and \cite{Vazquez_smoothing_effects} for a more comprehensive analysis.

Mostly motivated by applications in numerical analysis and sampling methods, but also by the need of detecting \eqref{eq:main_PME} as a \enquote{good approximation} of a discrete system of interacting agents in some contexts, increasing attention has recently been devoted to the approximation of nonlinear transport/diffusion models via (a suitable combination/interpolation of) finitely many \emph{moving particles}, obeying a suitable system of ODEs. Such a system is typically obtained as the discretization of the Lagrangian formulation of the underlying model. More specifically for  \eqref{eq:main_PME} regarded as a continuity equation
\[\partial_t\rho +\mathrm{div}(\rho v)=0\,,\qquad v=-\frac{m}{m-1}\nabla \rho^{m-1}\,,\]
the moving particles $[0,+\infty)\ni t\mapsto x(t)\in \R^d$ should somehow satisfy a law that approximates
\begin{equation}\label{eq:particles_law}
    \dot{x}(t)=v(x(t),t)\,.
\end{equation}
The initial data for the particles' positions are computed after a \emph{sampling}, or \emph{atomization} of the initial datum $\overline{\rho}$ of the Cauchy problem for \eqref{eq:main_PME}. Based on the positions of the particles at time $t>0$, an approximate moving density is constructed, with the goal of proving that it is close to the solutions to \eqref{eq:main_PME} with $\overline{\rho}$ as initial condition in a suitable topology. Since the approximation is constructed by the law \eqref{eq:particles_law} following the path of \enquote{material particles} (with a terminology borrowed from classical fluid mechanics), this whole procedure is also referred to as \emph{Lagrangian} particles approximation. As particles obey a deterministic law, the above approximation procedure is often called Deterministic Particle Approximation (DPA).

A common DPA for \eqref{eq:main_PME} is obtained by suitably convolving the above velocity field $v$, for example by considering $v_\varepsilon=\frac{m}{m-1}\nabla K_\varepsilon\ast \rho^{m-1}$ with $K_\varepsilon$ a suitable symmetric, often smooth kernel, and $\varepsilon>0$ a scaling parameter vanishing as the number of particles tends to infinity, with $K_\varepsilon\rightarrow \delta_0$ in some measure sense as $\varepsilon\searrow 0$, see for example \cite{oelschlaeger} for the stochastic case, \cite{carrillo_craig_patacchini,burger_esposito,DF_iorio_schmidtchen} for the fully deterministic one. This method, often referred to as \enquote{blob method}, while it has the clear advantage of easing the definition of the ODE system for the moving particles (in that the velocity field may be calculated pointwise even if $\rho$ is discontinuous, provided $K_\varepsilon$ is smooth enough), essentially turns the diffusion operator in \eqref{eq:main_PME} into a nonlocal transport operator, thus drastically changing the structure of the particles' dynamics in the case of initial data with \enquote{high concentrations}. Such an issue is reflected in the particles ODEs of a blob method, in which every particles actually interacts with any other particle, although with weights that vanish in the limit for particles very far from each other.

A perhaps more natural way of reproducing \eqref{eq:main_PME} at the level of Lagrangian particles and maintaining, at the same time, the features of a local model, is to restrict to \emph{nearest neighbor} interactions, that is, to allow every particle to interact only with \enquote{nearby} particles. Clearly, the feasibility of this method is highly dependent on the spatial dimension. In one space dimension, taking $N$ ordered particles with positions $x^{(N)}_0<x^{(N)}_1<\ldots<x^{(N)}_N$, the nearest neighbor particles of a given $x^{(N)}_i$ are naturally identified with $x^{(N)}_{i-1}$ and $x^{(N)}_{i+1}$.  Following the seminal idea of \cite{russo}, in which a discrete analog of the velocity in \eqref{eq:particles_law} was identified (and referred to as \enquote{osmotic velocity}) for the linear diffusion equation, a nearest neighbor interaction approximation for the one-dimensional \eqref{eq:main_PME} was considered in \cite{Gosse}, which investigated the approximation problem mainly from a numerical point of view. The corresponding particle system reads
\begin{equation}\label{eq:particles_intro}
    \dot{x}^{(N)}_i(t)=-N(R^{(N)}_{i+1}(t)^m-R^{(N)}_{i}(t)^m)\,,\qquad i\in\{0,\ldots,N\}
\end{equation}
where
\[R^{(N)}_i(t)=\frac{1}{N(x^{(N)}_{i}(t)-x^{(N)}_{i-1}(t))}\,,\qquad i\in\{1,\ldots,N\}\,.\]
Here, $x^{(N)}_i(t)$ with $i\in\{0,\ldots,N\}$ are $N+1$ ordered particles, and $R^{(N)}_{0}(t)=R^{(N)}_{N+1}(t)=0$ by convention. We shall provide an informal justification for \eqref{eq:particles_intro} in Subsections \ref{subsec:lagrangian} and \ref{subsec:statement_AB}. In this paper, as well as in other papers on similar subjects (see e.g. \cite{DF_rosini} for scalar conservation laws), the following discrete reconstruction of the density
\begin{equation}\label{eq:discrete_density_intro}
  \rho^{(N)}(t,x)=\sum_{i=1}^{N}R^{(N)}_i(t)\mathbf{1}_{[x^{(N)}_{i-1}(t),x^{(N)}_{i}(t))}(x)
\end{equation}
is used for comparison with the solution $\rho$ to the continuum model \eqref{eq:main_PME}.

Rigorous results on the convergence of $\rho^N$ towards weak solutions to
\begin{align}
    \label{eq:PME}
    \rho_t = (\rho^m)_{xx}
\end{align}
were obtained for the first time in \cite{Osberger} for a more general class of $1d$ models which includes  \eqref{eq:main_PME} as a special case. Further results have been obtained in \cite{fagioli_radici,daneri_radici_runa,daneri_radici_runa_2} for aggregation-diffusion equations. While the above particle scheme \eqref{eq:particles_intro} has the advantage of reproducing, at the discrete particle level, the variational formulation of \eqref{eq:main_PME} in the Wasserstein gradient flow setting of \cite{otto,AGS}, so far the identification of the limit as a weak solution to \eqref{eq:main_PME} has been possible only in frameworks in which, simplistically speaking, the solution to \eqref{eq:main_PME} is far from zero. Such an issue is due to the difficulty encountered so far in controlling the speed of propagation of the support of the approximating density reconstruction $\rho^N$ defined in \eqref{eq:discrete_density_intro} \emph{uniformly w.r.t. the number of particles.}

Partly related to that, the particle scheme \eqref{eq:particles_intro} has never been proven, so far, to satisfy a discrete counterpart of any of the \emph{regularizing effects} that characterize the porous medium equation, the most celebrated one being the Aronson-Bénilan fundamental estimate \cite{AB} (see also \cite[Proposition 9.4]{Vazquez}), which in the one-dimensional case states that, for a given nonnegative initial condition in $L^1(\R)$, at any $t>0$ the profile $\rho(t,\cdot)$ solving \eqref{eq:PME} is positive and smooth inside its support,
and satisfies there
\begin{align}
    \label{eq:AB}
    \partial_{xx}\left(\frac{\rho^{m-1}}{m-1}\right) \ge - \frac1{m(m+1)t}.
\end{align}
The estimate \eqref{eq:AB} is one of the key tools in the analysis of non-negative finite mass solutions $\rho$
to \eqref{eq:PME}. In fact, an existence theory of $L^1$ solutions to \eqref{eq:PME} can be carried out mostly from \eqref{eq:AB}, which provides higher order estimates needed to ensure compactness. The inequality \eqref{eq:AB} is sharp, and is saturated for the self-similar solution $\hat\rho$ to \eqref{eq:PME},
given by
\begin{align}
  \label{eq:PMEss}
  \hat\rho(t,x) = t^{-1/(m+1)}B\big(t^{-1/(m+1)}x\big)
  \quad \text{with} \quad
  B(\xi) = \left(C_m^{m-1} - \frac{m-1}{2m(m+1)}\xi^2\right)_+^{1/(m-1)},
\end{align}
where $C_m>0$ is a mass parameter. On grounds of \eqref{eq:AB} and elementary properties like mass conservation alone,
one can easily deduce a variety of qualitative features of solutions. One, which applies for compactly supported initial data, is the upper bound on the measure of the support at time $t>0$
\begin{equation}\label{eq:PME_support}
    \mathrm{meas}(\mathrm{supp}(\rho(\cdot,t)))\leq a + b t^{\frac{1}{m+1}}
\end{equation}
for some constants $a,b>0$ depending on the initial support and on $m$, see for example \cite{Vazquez_revista_1998}. Another one is the uniform in time decay of $\rho(t,\cdot)$,
\begin{align}
  \label{eq:PMELinfty}
  \max_x \rho(t,x) \le C_m t^{-1/(m+1)} \quad \text{for all $t>0$},
\end{align}
where $C_m$ is the mass parameter from \eqref{eq:PMEss} above, see \cite[Section 9.4]{Vazquez}. Both bounds \eqref{eq:PME_support} and \eqref{eq:PMELinfty} are sharp, as they are saturated by the self-similar solution $\hat\rho$ in \eqref{eq:PMEss}.

\medskip
The main goal of the paper at hand is to replicate the Aronson-Bénilan estimate \eqref{eq:AB}, as well as its consequences \eqref{eq:PME_support} and  \eqref{eq:PMELinfty},
for the discrete density $\rho^N$ defined in \eqref{eq:discrete_density_intro} out of the
particle discretization \eqref{eq:particles_intro} of the PME \eqref{eq:PME}.
Moreover, we use our discrete version of \eqref{eq:AB} to show rigorous convergence of the density \eqref{eq:discrete_density_intro} constructed from the scheme \eqref{eq:particles_intro} towards the weak solution to \eqref{eq:PME} in the full generality of $L^1(\R)$ initial data. The main conclusion of this paper is that a discrete version of \eqref{eq:AB} alone is enough to prove convergence of the scheme \eqref{eq:particles_intro}  towards \eqref{eq:PME} as well as to replicate qualitative properties such as \eqref{eq:PME_support} and  \eqref{eq:PMELinfty}. The concept of solution to \eqref{eq:PME} we consider here is that of distributional solutions in $L^\infty_{\mathrm{loc}}((0,+\infty)\,;\,L^\infty(\R))$ with distributional initial trace in $L^1(\R)$, for which uniqueness holds, see \cite{pierre_nonlinear_analysis_1982}.

The paper is structured as follows.
\begin{itemize}
    \item In Section \ref{sec:preliminaries} we introduce the main concepts, provide a heuristic derivation of the Aronson-Bénilan estimate \eqref{eq:AB}, set up our particle scheme and our  main assumptions, and state our results in Subsection  \ref{subsec:other_statements}. Our main results are collected in Theorems \ref{thm:main_AB}, \ref{thm:support_main}, \ref{thm:Linfty_main}, and \ref{thm:convergence_main} below.
    \item In Section \ref{sec:proof_estimate} we prove some basic properties of the particle scheme and provide the rigorous proof of the discrete Aronson-Bénilan estimate (that is, Theorem \ref{thm:main_AB}).
    \item In Section \ref{sec:support} we prove Theorem \ref{thm:support_main} on the uniform growth estimate of the support.
    \item In Section \ref{sec:decay} we prove the uniform $L^\infty$ decay property of Theorem \ref{thm:Linfty_main}.
    \item Finally, in Section \ref{sec:convergence} we prove our convergence result stated in Theorem \ref{thm:convergence_main}.
\end{itemize}

We conclude this introductory section setting some ground rules on notations. All $N$-depending approximating quantities will carry their  dependence on $N$ in different ways:
\begin{itemize}
    \item $N$-depending \emph{spaces} of vectors/functions will be denoted with $N$ as a subscript, with the only exception of the Euclidean space $\R^N$.
    \item $N$-depending \emph{maps} defined on linear spaces, or functional spaces, are denoted with $N$ as a superscript (e.g. a sampling map $\sam^N$).
    \item Finite dimensional vectors, for example in $\R^{N+1}$, are denoted in general as $X=(x_0,\ldots,x_N)$, with upper case letters to denoted the vector, lower case ones to denote their components. No dependence on the dimension is encoded in the notation. However, when dealing with \emph{solutions to a particle scheme}, their dependence on $N$ is (often) emphasized with the superscript notation $X^{(N)}$.
\end{itemize}

\section{Preliminary concepts and statement of the main results}\label{sec:preliminaries}

\subsection{Probability measures, pseudo-inverses, Lagrangian formulation}\label{subsec:lagrangian}

Let $\mathcal{P}(\R)$ be the set of probability measures on $\R$. For a given $\rho\in \mathcal{P}(\R)$, we define the cumulative distribution function of $\rho$, denoted by $F_\rho:\R\rightarrow [0,1]$, as
\[F_\rho(x)=\rho((-\infty,x])\,,\qquad x\in \R\,,\]
and its right-continuous pseudo inverse $X_\rho:(0,1)\rightarrow \R$ as
\begin{equation}\label{eq:mapT1}
    X_\rho(z)=\inf\left\{x\in \R\,:\,\, F_\rho(x)>z\right\}\,.
\end{equation}
Given
\[\mathcal{K}=\left\{X:(0,1)\rightarrow \R\,:\,\,\hbox{$X$ is right continuous and non-decreasing}\right\}\]
the map $\mathcal{T}:\mathcal{P}(\R)\rightarrow \mathcal{K}$ defined by
\begin{equation}\label{eq:mapT2}
    \mathcal{P}(\R)\ni\rho\mapsto \mathcal{T}[\rho]=X_\rho\in \mathcal{K}
\end{equation}
is a bijection. Assuming the necessary regularity for a solution $\rho$ to  \eqref{eq:PME}, a Lagrangian formulation satisfied by the function
\begin{equation}\label{eq:pseudo}
    X(t,z)=X_{\rho(t,\cdot)}(z)=\mathcal{T}[\rho(t,\cdot)](z)\,,
\end{equation}
can be recovered (see e.g. \cite{carrillo_toscani_wasserstein_PME}), namely
\begin{equation}\label{eq:lagrangian_PME}
    \partial_t X(z,t)=-\partial_z\left(\frac{1}{\partial_z X(z,t)}\right)\,.
\end{equation}
More precisely, \eqref{eq:lagrangian_PME} follows rigorously by assuming $\rho$ solves \eqref{eq:PME}, with $\rho$ smooth on its support and and the set $\left\{x\in \R\,:\,\,\rho(x,t)>0\right\}$ being a connected interval for all $t>0$.

For future use, we recall that the $1$-Wasserstein distance between two probability measures $\rho,\eta\in \mathcal{P}(\R)$ can be computed as
\begin{equation}\label{eq:wasserstein}
    d_1(\rho,\eta)=\|\mathcal{T}[\rho]-\mathcal{T}[\eta]\|_{L^1((0,1))}\,.
\end{equation}

\subsection{Lagrangian version of the Aronson-Bénilan estimate}\label{subsec:heuristics}

The Lagrangian formulation \eqref{eq:lagrangian_PME} of \eqref{eq:PME} appears to be particularly well-adapted to proving, at least formally, the estimate \eqref{eq:AB}. Assuming the necessary regularity on a solution $\rho$ to \eqref{eq:PME} and unit mass for simplicity, let $X$ be defined as in \eqref{eq:pseudo}. Let
\begin{equation}\label{eq:R}   R(t,z)=\rho(t,X(t,z))=\frac{1}{\partial_z X(t,z)}\,,
\end{equation}
which is finite for all $(t,z)\in (0,+\infty)\times (0,1)$ due to the regularity of $\rho$ and assuming that the support of $\rho(\cdot,t)$ is a connected interval. Formally, thanks to \eqref{eq:lagrangian_PME}, $R$ satisfies
% \eqref{eq:PME} in terms of the inverse distribution function $X$ and the associated \enquote{reshuffled} density
% \begin{align}
%     \label{eq:R}
%     R := \rho\circ X = \frac1{X'}.
% \end{align}
% The result is
\begin{align}
    \label{eq:LPME}
    \partial_t R(t,z) = R(t,z)^2\partial^2_{zz}(R(t,z)^m).
\end{align}
Now rewrite the second derivative on the left-hand side of \eqref{eq:AB} with the help of \eqref{eq:R}. Since
\begin{align*}
    \frac{m\rho^{m-1}(t,X(t,z))}{m-1} = \frac{mR^{m-1}(t,z)}{m-1},
\end{align*}
we obtain for the first derivative
\begin{align*}
    \left(\frac{m\rho^{m-1}}{m-1}\right)_x\circ X\,X_z = mR^{m-2}R_z\,,
\end{align*}
which implies
\begin{align*}
    &  \left(\frac{m\rho^{m-1}}{m-1}\right)_x\circ X = (R^m)_z.
\end{align*}
Accordingly, for the second derivative we obtain
\begin{align*}
    \left(\frac{m\rho^{m-1}}{m-1}\right)_{xx}\circ X X_z = (R^m)_{zz}\,,
\end{align*}
which implies
\begin{align*}
    &   \left(\frac{m\rho^{m-1}}{m-1}\right)_{xx}\circ X = R(R^m)_{zz}.
\end{align*}
Thus, the quantity of interest in \eqref{eq:AB} in Lagrangian variables is
\begin{align}
    \label{eq:LZ}
    Z(t,z):=R(t,z)(R^m(t,z))_{zz}.
\end{align}
\begin{proposition}
    Consider a sufficiently smooth solution to the Lagrangian formulation \eqref{eq:LPME} of the porous medium equation. Then the function $Z$ defined in \eqref{eq:LZ} satisfies
    \begin{align}
        \label{eq:LAB}
        Z(t,z) \ge -\frac1{(m+1)t}.
    \end{align}
\end{proposition}
The meaning of \enquote{sufficiently smooth} above is intentionally vague. For example, the role of boundary conditions is neglected. In any case, the proof below is just for illustration of the principle, which will be made fully rigorous for spatially discrete solutions later.
\begin{proof}
    For the time derivative of $Z$, we obtain by means of \eqref{eq:LPME} the expression
    \begin{align*}
        \partial_t Z
        &= \partial_tR\,(R^m)_{zz} + mR(R^{m-1}\,\partial_tR)_{zz}\\
        &= \big[R(R^m)_{zz}\big]^2 + mR\big(R^{m+1}(R^m)_{zz}\big)_{zz} \\
        &= Z^2 + m R(R^{m}Z)_{zz} \\
        &= (m+1) Z^2 + 2m R(R^{m-1})_{z}Z_{z} + mR^mZ_{zz}.
    \end{align*}
    Now, at some fixed time $t>0$, consider a minimal point $\xi^*$ of $Z$. Then $Z_z(t,\xi^*)=0$ and $Z_{zz}(t,\xi^*)\ge0$. Consequently,
    \begin{align*}
     \partial_t Z(t,\xi^*) \ge (m+1) Z(t,\xi^*)^2.
    \end{align*}
    In other words, the minimum $f(t):=\min_\xi Z(t,\xi)$ of $Z$ satisfies the ODE
    \begin{align}
        \label{eq:ode}
        \dot f \ge (m+1) f^2,
    \end{align}
    which implies, with a simple ODE argument after separating the variables $f$ and $t$,
    \begin{align}
        \label{eq:est}
        f(t) \ge -\frac1{(m+1)t},
    \end{align}
    which yields \eqref{eq:LAB}.
\end{proof}

%%%%%%%%%%%%%%%%%%%%%%%%%%%%%%%%%%%%%%%%%%%%%%%%%%%%%%%
\subsection{Main assumptions and definition of the particle scheme}\label{subsec:scheme}
%%%%%%%%%%%%%%%%%%%%%%%%%%%%%%%%%%%%%%%%%%%%%%%%%%%%%%%
%
Throughout the whole paper we shall assume the initial condition $\overline{\rho}$ satisfies
\begin{align}
  \label{eq:In}
  \overline{\rho}\in \mathcal{P}(\R)\cap L^1(\R).
\end{align}
Fix a number $N\in \mathbb{N}$ of intervals, each carrying a mass of $1/N$.
Elements $X$ in the convex cone
\[
  \mathcal{K}_N=\left\{X=(x_0,\ldots,x_N)\in \mathbb{R}^{N+1}\,:\,\, x_i\leq x_{i+1}\ \text{for all $i=0,\ldots,N-1$}\right\}\, ,
\]
determine the $N+1$ end points of the intervals.
The total length $x_N-x_0$ of the collection of intervals is denoted by % $\mathcal{L}^N:\mathcal{K}_N\rightarrow [0,+\infty)$
\begin{equation}\label{eq:functional_L}
    \mathcal{L}^N[X]=x_N-x_0\,. %\qquad X=(x_0,\ldots,x_N)\in \mathcal{K}_N\,.
\end{equation}
We think as the mass portions on $1/N$ as uniformly distributed on each interval, i.e.,
the interval $(x_{i-1},x_i)$ carries a constant spatial density of $1/[N(x_i-x_{i-1})]$.
These $N$ density values are summarized by
the \emph{reconstruction map} $\rmap^N:\mathcal{K}_N\rightarrow\Rp^{N}$,
\begin{align*}
  & \rmap^N[X]=\big(\rmap^N_1[X],\ldots,\rmap^N_N[X]\big)
  \\
  & \ \text{with}\quad
    \rmap^N_i[X]=\frac{1}{N(x_{i}-x_{i-1})}\,,\quad i\in\{1,\ldots,N\}\,.
    % \quad \text{for $X=(x_0,\ldots,x_N)\in \mathcal{K}_N$},
\end{align*}
where $\Rp^N$ denotes the cone of $N$-vectors with non-negative entries.
Alternatively, the mass distribution can be reconstructed
as a piecewise constant interpolation $\rec^N:\mathcal{K}_N\rightarrow \mathcal{P}(\R)$:
\begin{equation}\label{eq:reconstructions}
  \rec^N[X]=\sum_{i=1}^{N}\rmap^N_i[X]\mathbf{1}_{[x_{i-1},x_{i})}\,.
  % \qquad \hbox{for $X=(x_0,\ldots,X_N)\in \mathcal{K}_N$}\,.
\end{equation}
% Here, $\mathcal{P}(\mathbb{R})$ is the space of probability measures.
For later reference, we further define $\mathcal{D}^N:\mathcal{K}_N\rightarrow\Rp^N$ by
\begin{align*}
  \mathcal{D}^N[X]=\big(\mathcal{D}^N_1[X],\ldots,\mathcal{D}^N_{N}[X]\big)
  \quad\text{with}\quad
  \mathcal{D}^N_i[X]=(x_{i}-x_{i-1}),\,\quad X=(x_0,\ldots,x_N)\in \mathcal{K}_N\,,
\end{align*}
which implies  $\rmap^N_i[X]=1/(N\,\mathcal{D}^N_i[X])$.

By a \emph{sampling map}, we mean an association of any absolutely continuous probability measure on $\R$
to an ordered $(N+1)$-vector of nonoverlapping particle positions,
\begin{equation}\label{eq:initial_sampling_2}
  \sam^N:\mathcal{P}(\mathbb{R})\cap L^1(\mathbb{R})\rightarrow \mathcal{K}_N^\circ\,,\quad \sam^N[\rho]=\left(\sam^N_0[\rho],\ldots,\sam^N_N[\rho]\right)
\end{equation}
with the approximation property
\begin{equation}\label{eq:initial_sampling}
  \frac{1}{N+1}\sum_{i=0}^{N}\varphi\left(\sam^N_i[\rho]\right)\rightarrow \int_\R \varphi(x)\rho(x) dx\qquad \hbox{as $N\rightarrow+\infty$, for all $\varphi\in C_0(\R)$}\,.
\end{equation}
% We also require as a standing assumption
% \begin{equation}\label{eq:initial_sampling_2}
%   \sam^N\left(\mathcal{P}(\R)\cap L^1(\R)\right)\subset \mathcal{K}_N^\circ\,.
% \end{equation}
% The above map will be used to sample a given initial condition in $\mathcal{P}(\R)\cap L^1(\R)$ with $N+1$ ordered points of the real line.
% Assumption \eqref{eq:initial_sampling_2} prescribes that initial sampled particles cannot overlap.
In some cases we shall also require the additional assumption
\begin{equation}\label{eq:sampling_support_N}
   \lim_{N\rightarrow +\infty} \frac{\mathcal{L}^N[\sam^N[\rho]]^2}{N} = 0\,,\qquad \hbox{for all $\rho\in \mathcal{P}(\mathbb{R})\cap L^1(\mathbb{R})$}\,.
\end{equation}
There are a variety of possible sampling maps with the above properties.
A frequently used one, which we denote by $\sam^N_*$,
applies to compactly supported measures $\rho\in \mathcal{P}(\R)\cap L^1(\R)$ and is defined as follows.
First set
\begin{equation}\label{eq:sampling_support_0}
    \sam^N_{*,0}[\rho]=\inf{\mathrm{supp}(\rho)},
\end{equation}
and then recursively for $i=1,\ldots,N$:
\begin{equation}\label{eq:sampling_support_i}
   \sam^N_{*,i}[\rho]=\inf\left\{x\geq \sam^N_{*,i-1}[\rho]\,:\,\, \int_{\sam^N_{*,i-1}[\rho]}^x d\rho\geq \frac{1}{N}\right\}\,.
\end{equation}
We observe that $\sam^N_*[\rho]\in \mathcal{K}^\circ$ and that $\sam^N_{*,N}[\rho]=\sup \mathrm{supp}(\rho)$.
We shall refer to it as \emph{support preserving sampling}.

We are now ready to define our approximating scheme.
Let $\overline{\rho}$ satisfy \eqref{eq:In} and assume $\sam^N$ is a given sampling map.
% \eqref{eq:initial_sampling} and \eqref{eq:initial_sampling_2}.
Set
\begin{equation}\label{eq:initial_sampling_X}
  \overline{X}^{(N)}=(\overline{x}^{(N)}_0,\ldots,\overline{x}^{(N)}_N)=\sam^N[\overline{\rho}]\in \mathcal{K}_N
\end{equation}
and define the curve
\[[0,+\infty)\ni t\mapsto X^{(N)}(t)=(x^{(N)}_0(t),\ldots,x^{(N)}_N(t))\in \mathcal{K}_N\]
as the unique, local-in-time solution to the Cauchy problem
\begin{equation}\label{eq:particles1}
    \begin{dcases}
      \dot{x}^{(N)}_i(t)=-N(R^{(N)}_{i+1}(t)^m-R^{(N)}_{i}(t)^m) & i=1,\ldots,N-1\\
      \dot{x}^{(N)}_0(t)=-N R^{(N)}_1(t)^m & \\
      \dot{x}^{(N)}_N(t)=N R^{(N)}_{N}(t)^m & \\
      x^{(N)}(0)=\overline{X}^{(N)}\,, &
    \end{dcases}
\end{equation}
where we denoted $R^{(N)}(t)=(R^{(N)}_1(t),\ldots,R^{(N)}_N(t))$ the vector
of mass densities on the time-dependent intervals,
\begin{equation}
  R^{(N)}_i(t)=\rmap_i[X^{(N)}(t)]=\frac{1}{N(x^{(N)}_i(t)-x^{(N)}_{i-1}(t))}\,,\qquad i=1,\ldots,N\,.\label{eq:RN}
\end{equation}
With the convention that
\begin{align}
  \label{eq:Rconvent}
  R^{(N)}_0(t) = R^{(N)}_{N+1}(t)=0,
\end{align}
the first evolution equation in \eqref{eq:particles1} actually holds for all $N+1$ points, i.e., from $i=0$ to $i=N$.
% We observe that \eqref{eq:particles1} may be re-written as
% \begin{equation}\label{eq:particles1_differences}
%  \dot{x}_i(t)=-D^-_i \rmap[X(t)]^m\,,\qquad i\in \{0,\ldots,N\}
% \end{equation}
% with the convention
% \[\rmap_{-1}[X(t)]=\rmap_N[X(t)]=0\,.\]
We will also use the vector $d^{(N)}(t)=(d^{(N)}_1(t),\ldots,d^{(N)}_N(t))$ of interval lengths, given by
\begin{equation}\label{eq:def_di}
    d^{(N)}_i(t)=\mathcal{D}^N_i[X^{(N)}(t)]=(x^{(N)}_i(t)-x^{(N)}_{i-1}(t))\,,\qquad i\in\{1,\ldots,N\}\,,
\end{equation}
We finally introduce the quantity that is supposed to approximate the solution $\rho$ to \eqref{eq:PME} for large $N$, namely
\begin{equation}\label{eq:approximate_density}
    \rho^{(N)}(t,x)=\rec^N[X(t)]=\sum_{i=1}^N R^{(N)}_i(t)\mathbf{1}_{[x^{(N)}_{i-1}(t),x^{(N)}_i(t))}(x)\,.
\end{equation}

%%%%%%%%%%%%%%%%%%%%%%%%%%%%%%%%%%%%%%%%%%%%%%%%%%%%%%%
\subsection{Finite difference operators}\label{subsec:finite}
%%%%%%%%%%%%%%%%%%%%%%%%%%%%%%%%%%%%%%%%%%%%%%%%%%%%%%%
%
For vectors $F=(F_1,\ldots,F_{N})\in\R^N$, we define
the forward/backward difference quotients $\dff^+F$/$\dff^-F$ and the discrete Laplacian $\Delta F$ in the usual way,
\begin{align*}
  \dff^+_kF = N(F_{k+1}-F_k),
  \quad
  \dff^-_kF = N(F_{k}-F_{k-1}),
  \quad
  \Delta_kF = N^2(F_{k+1}+F_{k-1}-2F_k) = \dff^-_k\dff^+F,
\end{align*}
for all $k\in\Z$,
where $F_j$ for indices $j$ outside of $\{1,\ldots,N\}$ is interpreted as zero
(so, for example, $\dff^-_1 F = NF_1$, $\Delta_N F=N^2(F_{N-1}-2F_N)$, $\Delta_{0} F =N^2F_1$).
This definition automatically implies the following summation by parts rule for arbitrary $F,G\in\R^N$:
\begin{align}
  \label{eq:byparts}
  -\sum_{k=1}^{N}G_k\Delta_kF = \sum_{k=0}^N\dff^+_kF\dff^+_kG.
\end{align}
We also introduce the notation
\[F^\alpha =(F_0^\alpha,\ldots,F_{N-1}^\alpha)\]
for $\alpha>0$ and for vectors $F=(F_1,\ldots,F_{N})\in \Rp^N$.

We observe that, with the above notation and the convention \eqref{eq:Rconvent},
the evolution equations in \eqref{eq:particles1} may be re-written in the concise form
\begin{equation}\label{eq:particles1_differences}
 \dot{x}^{(N)}_i(t)=-\dff^+_i \rmap^N[X^{(N)}(t)]^m=-\dff^+_i R^{(N)}(t)^m\,,\qquad i\in \{0,\ldots,N\}
\end{equation}

The following technical lemma will be used later on in the paper.
\begin{lemma}\label{lem:technical1}
    Let $F=(F_1,\ldots,F_{N})\in \Rp^N$. Then,
    \[\sum_{k=0}^{N+1} \Delta_k[F]=0\]
    Moreover, either $F_k=0$ for all $k=1,\ldots,N$ or there exists $k\in \{1,\ldots,N\}$ such that $\Delta_k F<0$.
\end{lemma}
The proof is provided in the Appendix \ref{sec:appendix_lemma}.

%%%%%%%%%%%%%%%%%%%%%%%%%%%%%%%%%%%%%%%%%%%%%%%%%%%%%%%
\subsection{Heuristics}\label{subsec:statement_AB}
%%%%%%%%%%%%%%%%%%%%%%%%%%%%%%%%%%%%%%%%%%%%%%%%%%%%%%%
%
In this subsection we provide a heuristic justification of both the particle scheme \eqref{eq:particles1} and of our discrete version of the Aronson-Bénilan inequality, and then state the latter as a theorem. For simplicity, assume in this subsection that the initial condition $\overline{\rho}$ is sampled via the support preserving map $\sam^N_*$ defined in \eqref{eq:sampling_support_0}-\eqref{eq:sampling_support_i}. Let $\overline{X}=\mathcal{T}[\overline{\rho}]$ with $\mathcal{T}$ the map defined in \eqref{eq:mapT2}. Assume that the system \eqref{eq:particles1} has a unique, global-in-time solution $X^{(N)}(t)=(x^{(N)}_0(t),\ldots,x^{(N)}_N(t))\in \mathcal{K}_N$, and consider its density reconstruction $\rho^{(N)}(t,x)=\rec^N[X^{(N)}(t)](x)$ in \eqref{eq:approximate_density}. One can easily see that the function $X^{(N)}:[0,+\infty)\times[0,1)\rightarrow \R$ defined by
\[X^{(N)}(t,z)=\sum_{i=0}^{N-1}\left[x^{(N)}_i(t)+R^{(N)}_{i+1}(t)^{-1}\left(z-\frac{i}{N}\right)\right]\mathbf{1}_{[i/N,(i+1)/N)}(z)\,,\]
restricted to $z\in(0,1)$, satisfies
\[X^{(N)}(t,z)=\mathcal{T}[\rho^{(N)}(t,\cdot)](z)\,,\]
where $\mathcal{T}$ is once again the map \eqref{eq:mapT2}.
Moreover, the function
\[R^{(N)}(t,z)=\sum_{i=0}^{N-1}R^{(N)}_{i+1}(t)\mathbf{1}_{[i/N,(i+1)/N)}(z)\]
satisfies
\[R^{(N)}(t,z)=\rho^{(N)}(t,X^{(N)}(t,z))\,.\]
We may interpret $(0,1)\ni z\mapsto X^{(N)}(t,z)$ as a linear interpolation of the approximation on the fixed grid $\{i/N\,:\,\,i=0\ldots,N\}$ of the function $(0,1)\ni z \mapsto X(t,z)=\mathcal{T}[\rho(t,\cdot)](z)$ with $\rho$ solving \eqref{eq:PME}. Since we expect $X$ to solve \eqref{eq:lagrangian_PME}, it is quite natural to replace the partial derivative $\partial_z$ by a finite difference, for example $\dff^+_i$, to recover exactly our particle scheme \eqref{eq:particles1_differences}, which is equivalent to \eqref{eq:particles1}. Moreover,
$R$ defined in \eqref{eq:R} is expected to solve \eqref{eq:LPME}. Hence, we naturally expect $R^{(N)}$ to solve a finite-difference version of \eqref{eq:LPME}. Given that $R^{(N)}(t,\cdot)$ is piecewise constant with values $R^{(N)}_1(t),\ldots,R^{(N)}_{N}(t)$, a natural discrete counterpart of the quantity $Z$ defined in \eqref{eq:LZ} is
\begin{equation}\label{eq:dZ}
    Z^{(N)}(t,z)=\sum_{k=1}^{N}Z^{(N)}_k(t)\mathbf{1}_{[i/N,(i+1)/N)}(z)\,,\qquad Z^{(N)}_k(t)=R^{(N)}_k(t)\Delta_k R^{(N)}(t)^m\,.
\end{equation}
Based on the above considerations, a totally reasonable discrete Lagrangian version for the Aronson-Bénilan estimate \eqref{eq:AB} is
\[ Z^{(N)}_k(t)\geq -\frac{1}{(m+1)t}\,,\qquad k\in \{1,\ldots,N\}\,,\quad t>0\,,\]
which we will state in a more precise way in the next subsection.

%%%%%%%%%%%%%%%%%%%%%%%%%%%%%%%%%%%%%%%%%%%%%%%%%%%%%%%
\subsection{Statement or our main results}\label{subsec:other_statements}
%%%%%%%%%%%%%%%%%%%%%%%%%%%%%%%%%%%%%%%%%%%%%%%%%%%%%%%
%
First, we will prove in Section \ref{sec:proof_estimate} that the particle scheme \eqref{eq:particles1} and the reconstruction $\rho^{(N)}$ defined in \eqref{eq:approximate_density} are uniquely and globally defined. We then state the main result of this paper, that is, our discrete version of the Aronson-Bénilan estimate.
\begin{theorem}[Discrete Aronson Bénilan estimate]\label{thm:main_AB}
  The quantity $Z^{(N)}$ defined in \eqref{eq:dZ} satisfies
  \begin{equation}\label{eq:dAB_initial}
    \overline{Z}^{(N)}=\min_{k\in \{1,\ldots,N\}}Z^{(N)}_k(0)<0\,.
  \end{equation}
  Moreover, for all $t\geq 0$,
  % \begin{equation}\label{eq:dAB}
  %   Z_k(t)\geq -\frac{1}{(m+1)t}\,,\qquad \hbox{for all $k\in\{1,\ldots,N-1\}$}\,.
  % \end{equation}
  % Moreover, if
  % \[\overline{Z}=\min_{i=0,\dots,N-1}Z_i(0)<0\]
  % then
  \begin{equation}\label{eq:dAB2}
    Z^{(N)}_k(t)\geq -\frac{1}{|\overline{Z}^{(N)}|^{-1}+(m+1)t}\,,\qquad \hbox{for all $k\in\{1,\ldots,N\}$}\,.
  \end{equation}
\end{theorem}
The proof of Theorem \ref{thm:main_AB} is contained in Section \ref{sec:proof_estimate}.

Besides its use to prove convergence of the scheme to a solution to the PME \eqref{eq:PME}, Theorem \ref{thm:main_AB} may be applied to reproduce some important qualitative properties of solutions to \eqref{eq:PME} at the level of the particle scheme. The first one is an estimate of the speed of propagation of the support of $\rho^{(N)}$. It is the discrete counterpart of the estimate \eqref{eq:PME_support}. $\mathcal{L}^N$ below is the functional defined in \eqref{eq:functional_L}.
\begin{theorem}[Speed of propagation of the support]\label{thm:support_main}
  There is a constant $B$, only depending on $m>1$, such that any solution $X^{(N)}$ to the evolution equation \eqref{eq:particles1}
  satisfies
  \begin{equation}\label{eq:Lest}
    \mathcal{L}^N[X^{(N)}(t)]\leq \mathcal{L}^N[X^{(N)}(0)] + Bt^{\frac{1}{m+1}}\qquad \hbox{for all $t\geq 0$}\,.
  \end{equation}
\end{theorem}
Theorem \ref{thm:support_main} will be proven in Section \ref{sec:support}. The second qualitative property as a consequence of Theorem \ref{thm:main_AB} is an $L^\infty$-estimate on discrete solutions that parallels \eqref{eq:PMELinfty}.
With the approach that we take, we are not able to replicate the optimal constant $C_m$,
but we derive that analogous estimate with another universal (and in particular discretization-independent) constant.
\begin{theorem}[$L^\infty$ decay]
  \label{thm:Linfty_main}
  The quantities $R^{(N)}_k$ defined in \eqref{eq:RN} satisfy at every $t>0$
  \begin{align}
    \label{eq:Linfty}
    \max_{k\in \{1,\ldots,N\}} R^{(N)}_k(t) \le \left(\frac{m+1}{16m\,t}\right)^{1/(m+1)}.
  \end{align}
\end{theorem}
Theorem \ref{thm:Linfty_main} will be proven in Section \ref{sec:decay}.

Finally, another important application of Theorem \ref{thm:main_AB} is that it implies the convergence of the scheme in the full generality of $L^1$ initial data.

\begin{theorem}[Convergence of the scheme]
  \label{thm:convergence_main}
  Let $\overline{\rho}$ satisfy \eqref{eq:In},
  and assume the initial sampling map $\sam^N$ satisfies \eqref{eq:initial_sampling}  and \eqref{eq:sampling_support_N}.
  Let $X^{(N)}(t)$ be the solution to the scheme \eqref{eq:particles1} and let $\rho^{(N)}(t,x)$ be as in \eqref{eq:approximate_density}.
  Then, for all $T\geq 0$, $\rho^{(N)}$ converges in $L^1([0,T]\times \R)$ to a very weak solution $\rho\in L^1\cap L^m([0,T]\times \R)$ of the porous medium equation \eqref{eq:PME}
  with initial datum $\overline{\rho}$, i.e.,
  \begin{align}
    \label{eq:weakPME}
    \int_0^T \intR \big(\rho\,\partial_t\varphi+\rho^m\,\partial_{xx}\varphi\big)\dd x\dd t = -\intR\varphi(0,x)\overline{\rho}(x)\,dx\,,
  \end{align}
  for every test function $\varphi\in C^\infty_c([0,+\infty)\times \R)$.
\end{theorem}
Theorem \ref{thm:convergence_main} will be proven in Section \ref{sec:convergence}.

%%%%%%%%%%%%%%%%%%%%%%%%%%%%%%%%%%%%%%%%%%%%%%%%%%%%%%%
%%%%%%%%%%%%%%%%%%%%%%%%%%%%%%%%%%%%%%%%%%%%%%%%%%%%%%%
\section{Well posedness of the scheme and proof of Theorem \ref{thm:main_AB}}
\label{sec:proof_estimate}

In this section and in the next two, we will often lighten the notation and suppress the superscript $^{(N)}$ for the solution to the particle scheme \eqref{eq:particles1} and their related variables $R_k, d_k, Z_k$. We will maintain the superscript in the statements of all the results for clarity. The ODE system \eqref{eq:particles1} has a unique local-in-time solution, which may be proven to be global.

\begin{proposition}[Global Existence and Discrete Minimum Principle]\label{prop:MP}
Let $\overline{\rho}\in \mathcal{P}(\mathbb{R})\cap L^1(\mathbb{R})$ and let $\overline{X}^{(N)}$ be defined by \eqref{eq:initial_sampling_X}, with $\sam^N$ satisfying \eqref{eq:initial_sampling} and \eqref{eq:initial_sampling_2}. Then, the Cauchy problem \eqref{eq:particles1} with initial datum $\overline{X}^{(N)}$ has a unique, global-in-time solution. Moreover, the estimate
\begin{equation}\label{eq:MP}
        d^{(N)}_i(t)\geq \overline{d}^{(N)}=\min\left\{d^{(N)}_i[\overline{X}^{(N)}]\,:\,\, i=1,\ldots,N\right\}
    \end{equation}
    holds for all $t\geq 0$, with $d^{(N)}_i$ defined in \eqref{eq:def_di}.
\end{proposition}

\begin{proof}
    The assumption \eqref{eq:initial_sampling_2} ensures $\overline{x}_i<\overline{x}_{i+1}$ for all $i=0,\ldots,N-1$. Hence, from the classical Cauchy-Lipschitz theory, a unique solution $[0,T)\ni t\mapsto X(t)$ exists for a sufficiently small $T>0$. We prove \eqref{eq:MP} by a contradiction argument. Let $t^*\geq 0$ be the first time at which
    \[\min_{i=1,\ldots,N}d_i(t^*)=\overline{d}\]
    and
    \[\min_{i=1,\ldots,N}d_i(t^*)<\overline{d}\qquad \hbox{on the interval $(t^*,t^*+\varepsilon)$}\]
for some $\varepsilon>0$. By possibly choosing a smaller $\varepsilon$, there exists an index $i_0\in \{1,\ldots,N\}$ such that
\[d_{i_0}(t)=\min_{i=1,\ldots,N}d_i(t)\]
for $t\in [t^*,t^*+\varepsilon)$, which implies for $t\in (t^*,t^*+\varepsilon)$
\begin{align*}
    & d_{i_0}(t)=\overline{d}+\int_{t^*}^t\left(\dot{x}_{i_0}(s)-\dot{x}_{i_0-1}(s)\right) ds = \overline{d}-N\int_{t^*}^t\left(R_{i_0+1}(s)^m+R_{i_0-1}(s)^m-2R_ {i_0}(s)^m\right) ds,
\end{align*}
using the convention \eqref{eq:Rconvent}.
Now, for $t\in [t^*,t^*+\varepsilon)$ we have
\[R_{i_0}(t)=\max_{i=1,\ldots,N}R_i(t)\geq \max\left\{R_{i_0-1}(t),R_{i_0+1}(t)\right\}\,,\]
which implies
\[d_{i_0}(t)\geq\overline{d}\qquad \hbox{on $t\in (t^*,t^*+\varepsilon)$}\,,\]
which is a contradiction. We have therefore proven \eqref{eq:MP} on the short-time existence interval $[0,T)$. Global existence then follows by a standard continuation argument, and \eqref{eq:MP} follows for all times.
\end{proof}
As a consequence of Proposition \ref{prop:MP}, the approximated density $\rho^{(N)}$ in \eqref{eq:approximate_density}
is defined globally in time. From \eqref{eq:particles1} we easily obtain the ODE system for the discrete densities $R^{(N)}_k(t)=R_k(t)$
\begin{equation}\label{eq:particles_R}
\dot{R}_i(t)=
\begin{dcases}
N^2 R_1(t)^2(R_2^m(t)-2R_1^m(t)) & \hbox{if $i=1$} \\
N^2R_i(t)^2\left(R_{i+1}^m(t)+R_{i-1}^m(t)-2R_i^m(t)\right) & \hbox{if $i\in \{2,\ldots,N-1$\}}\\
N^2 R_{N}(t)^2(R_{N-1}^m(t)-2 R_{N}^m(t)) & \hbox{if $i=N$}
\end{dcases}
\end{equation}
which may be re-written as
\begin{equation}\label{eq:particles_R_2}
    \dot{R}_i(t)=R_i(t)^2 \Delta_i[R^m(t)]=R_i(t)Z_i(t)\,,\qquad i\in\{0,\ldots,N-1\}\,.
\end{equation}
We are now ready to prove Theorem \ref{thm:main_AB}.

% \begin{theorem}[Discrete Aronson Bénilan estimate]\label{prop:AB}
% There holds
% \begin{equation}\label{eq:dAB_initial}
%     \overline{Z}=\min_{k\in \{0,\ldots,N-1\}}Z_k(0)<0\,.
% \end{equation}
% Moreover, for all $t\geq 0$, the quantity $Z$ defined in \eqref{eq:dZ} satisfies
% % \begin{equation}\label{eq:dAB}
% %         Z_k(t)\geq -\frac{1}{(m+1)t}\,,\qquad \hbox{for all $k\in\{1,\ldots,N-1\}$}\,.
% %     \end{equation}
% % Moreover, if
% % \[\overline{Z}=\min_{i=0,\dots,N-1}Z_i(0)<0\]
% % then
% \begin{equation}\label{eq:dAB2}
%         Z_k(t)\geq -\frac{1}{|\overline{Z}|^{-1}+(m+1)t}\,,\qquad \hbox{for all $k\in\{0,\ldots,N-1\}$}\,.
%     \end{equation}
% \end{theorem}

\begin{proof}[Proof of Theorem \ref{thm:main_AB}]
  The inequality \eqref{eq:dAB_initial} follows from Lemma \ref{lem:technical1}.
  Differentiating $Z_k(t)$ in time, recalling the convention \eqref{eq:Rconvent}, we obtain
  \begin{align*}
    \dot Z_k
    &= \dot R_k\,\Delta_k R^m + N^2 m R_k\big(R_{k+1}^{m-1}\dot R_{k+1} + R_{k-1}^{m-1}\dot R_{k-1} - 2R_k^{m-1}\dot R_k) \\
    &= \big[R_k\Delta R_k^m\big]^2 + N^2 m R_k\big( R_{k+1}^{m+1}\Delta_{k+1} R^m + R_{k-1}^{m+1}\Delta_{k-1}R^m - 2R_k^{m+1}\Delta_k R^m \big) \\
    & = Z_k^2 + N^2 m R_k\big(R_{k+1}^m Z_{k+1} + R_{k-1}^m Z_{k-1} - 2R_k^m Z_k\big)\,.
  \end{align*}
  Fix $t>0$ and assume that $Z_k(t)$ attains its minimal value at $k=k^*$ on the interval $[t,t+\varepsilon)$. In particular, $Z_{k^*+1}(t)\ge Z_{k^*}(t)$ and $Z_{k^*-1}(t)\ge Z_{k^*}(t)$. Consequently, for all $k=0,\ldots,N-1$,
    \begin{align*}
        \dot Z_{k^*}(t)
        &\ge Z_{k^*}(t)^2 + N^2 m R_{k^*}(t)\big( R_{k^*+1}(t)^m + R_{k^*-1}(t)^m - 2R_{k^*}(t)^m\big) Z_{k^*}(t) \\
        &= Z_{k^*}(t)^2 + m R_{k^*}(t) \big(\Delta_{k^*}R(t)^m\big)Z_{k^*}(t) \\
        &= (m+1)Z_{k^*}(t)^2.
    \end{align*}
%     By integrating in time on $\tau\in (t,t+h]$, we get
%     \[
%     Z_k(\tau)\geq Z_k(t)+(m+1)\int_t^\tau Z_{k^*}(s)^2 ds\geq Z_{k^*}(t)+(m+1)\int_t^\tau Z_{k^*}(s)^2 ds
%     \]
%     and by taking the infimum with respect to $k$ on the left-hand side we obtain, for $\tau=t+h$,
%     \[
%     \frac{Z_{k^*}(t+h)-Z_{k^*}(t)}{h}\geq \frac{1}{h}\int_t^{t+h}Z_{k^*}(s)^2 ds\,.
%     \]
%     By letting $h\searrow 0$ we obtain
% \begin{equation}\label{eq:AB1}
%         \dot{Z}_{k^*}(t)\geq (m+1)Z_{k^*}(t)^2\,.
%     \end{equation}
Hence, the function
\[[0,+\infty)\ni t\mapsto Z(t)=\min_{k\in \{0,\ldots,N-1\}}Z_k(t)\]
satisfies the differential inequality
\begin{equation}\label{eq:AB2}
       \dot{Z}(t)\geq (m+1)Z(t)^2\,.
    \end{equation}
    A simple comparison argument for ODEs shows that $Z(t)\geq y(t)$,
    where
    \[y(t)=\frac{1}{\frac{1}{y(0)}-(m+1)t}\,.\]
    is the unique solution to the Cauchy problem
    \[\dot{y}(t)=(m+1)y(t)^2\,,\qquad y(0)=\overline{Z}\,.\]
    Since $y(0)<0$ we get the desired inequality \eqref{eq:dAB2}.
    % Thus, the minimum $f(t):=\min_k Z_k(t)$ satisfies again the ODE \eqref{eq:ode}, which implies \eqref{eq:est}, and thus shows \eqref{eq:dAB}.
\end{proof}

\section{Uniform growth estimate of the support}\label{sec:support}

%%%%%%%%%%%%%%%%%%%%%%%%%%%%%%%%%%%%%%%%%%%%%%%%%%%%%%%
%%%%%%%%%%%%%%%%%%%%%%%%%%%%%%%%%%%%%%%%%%%%%%%%%%%%%%%
%
% \textcolor{blue}{
%   The following result (which is used in the proof of convergence) should follow from the considerations here:
%   %
%   \begin{lemma}
%     \label{lem:Lest}
%     There is a universal constant $\Lambda>0$ such that the following is true
%     for every solution $x:[0,\infty)\to\R^{N+1}$ to \eqref{eq:particles1}:
%     \begin{align}
%       \label{eq:Lest}
%       x_N(t)-x_0(t) \le x_N(0)-x_0(0) + \Lambda t^{1/(m+1)} \quad \text{for all $t\ge0$}.
%     \end{align}
%   \end{lemma}
% }

In this section, we apply Theorem \ref{thm:main_AB} to detect a uniform bound for the distance
\[\mathcal{L}^{N}[X^{(N)}(t)]=x^{(N)}_N(t)-x^{(N)}_0(t)\]
with respect to the number of particles $N$ for some time $t\geq0$. The sampling map $\mathcal{S}^N$ for the initial condition $\overline{\rho}$ is kept general and satisfying \eqref{eq:initial_sampling} and \eqref{eq:initial_sampling_2}.

\begin{proposition}\label{eq:prop_support_1}
    For all $t\geq 0$ we have
    \begin{equation}\label{eq:estimate_support_1}
        \mathcal{L}^{N}[X^{(N)}(t)]\leq \mathcal{L}^{N}[X^{(N)}(0)](1+|\overline{Z}^{(N)}|(m+1)t)^{\frac{1}{m+1}}\,.
    \end{equation}
\end{proposition}

\begin{proof}
The ODE \eqref{eq:particles_R_2} and Theorem \ref{thm:main_AB} imply for $i=0,\ldots,N-1$
\begin{align*}
    & \dot{R}_i(t)\geq -\frac{R_i(t)}{|\overline{Z}|^{-1}+(m+1)t}\,.
\end{align*}
By a comparison argument for ODEs, $R_i(t)\geq u(t)$ with $u(t)$ being the only solution to the Cauchy problem
\[\dot{u}(t)=-\frac{u(t)}{|\overline{Z}|^{-1}+(m+1)t}\,,\qquad u(0)=R_i(0)\,,\]
which has the explicit solution
\[u(t)=R_i(0)(1+|\overline{Z}|(m+1)t)^{-\frac{1}{m+1}}\]
which in turns implies for all $t\geq 0$ and $i=0,\ldots,N-1$
\[R_i(t)\geq R_i(0)(1+|\overline{Z}|(m+1)t)^{-\frac{1}{m+1}}\,.\]
    Hence, we have the estimate
    \begin{align*}
        & \mathcal{L}^N[X(t)]=x_N(t)-x_0(t)=\sum_{i=0}^{N-1}(x_{i+1}(t)-x_i(t)) = \sum_{i=0}^{N-1}d_i(t)\\
        & \ = \sum_{i=0}^{N-1}\frac{1}{NR_i(t)}\leq \frac{1}{N}\sum_{i=0}^{N-1} \frac{1}{R_i(0)}(1+|\overline{Z}|(m+1)t)^{\frac{1}{m+1}} = \sum_{i=0}^{N-1}d_i(0)(1+|\overline{Z}|(m+1)t)^{\frac{1}{m+1}}\\
        & \ = \mathcal{L}^N[X(0)](1+|\overline{Z}|(m+1)t)^{\frac{1}{m+1}}\,.
    \end{align*}
\end{proof}
The estimate in Proposition \ref{eq:prop_support_1} is uniform with respect to $N$ only under the condition that the initial quantity $\overline{Z}^{(N)}$ is uniformly bounded from below with respect to $N$, which might not be the case if we are sampling an initial condition the space derivative of which has decreasing jumps. Our next goal is therefore to extend the above estimate to any initial datum satisfying \eqref{eq:In}. For a given particle trajectory $X^{(N)}(t)=(x^{(N)}_0(t),\ldots,x^{(N)}_N(t))$ we define
$\mathcal{X}_{X^{(N)}}:[0,1)\times [0,+\infty)\rightarrow \mathbb{R}$ by
\[\mathcal{X}_{X^{(N)}}(z,t)=\sum_{i=0}^{N-1}x^{(N)}_i(t)\mathbf{1}_{[i/N,(i+1)/N)}(z)\,.
\]
Given two particles trajectories $X^{(N)}(t)=(x^{(N)}_0(t),\ldots,x^{(N)}_N(t))$ and $Y^{(N)}(t)=(y^{(N)}_0(t),\ldots,y^{(N)}_N(t))$ and $h\in \mathbb{N}$, we set
\[d_h(X^{(N)}(t),Y^{(N)}(t))=\|\mathcal{X}_{X^{(N)}}(\cdot,t)-\mathcal{X}_{Y^{(N)}}(\cdot,t)\|_{L^{2h}((0,1))}\]
with the convention
\[d_\infty(X^{(N)}(t),Y^{(N)}(t))=\|\mathcal{X}_{X^{(N)}}(\cdot,t)-\mathcal{X}_{Y^{(N)}}(\cdot,t)\|_{L^{\infty}((0,1))}\,.\]
We prove the following result, which is a continuous-in-time version of a similar result in \cite{Gosse}. Although our result differs slightly from the result in \cite{Gosse}, the main idea behind the proof is the same.
\begin{proposition}\label{prop:contraction}
    The quantities $d_h(X^{(N)}(t),Y^{(N)}(t))$ for $h\in \mathbb{N}$ and $d_\infty(X^{(N)}(t),Y^{(N)}(t))$ are non increasing in time.
\end{proposition}
\begin{proof}
    For $h$ finite we compute
    \begin{align*}
        & \frac{d}{dt}d_h(X(t),Y(t))^{2h}=\frac{d}{dt}\int_0^1 \left(\mathcal{X}_{X}(z,t)-\mathcal{X}_{Y}(z,t)\right)^{2h}dz=\frac{1}{N}\sum_{i=0}^{N-1}\frac{d}{dt}\left(x_i(t)-y_i(t)\right)^{2h}\\
        & \ = \frac{2h}{N}\sum_{i=0}^{N-1}(x_i(t)-y_i(t))^{2h-1}(\dot{x}_i(t)-\dot{y}_i(t))\\
        & \ =  -2h\sum_{i=0}^{N-1}(x_i(t)-y_i(t))^{2h-1}\left(\left(R_{i+1}(t)^m-R_{i}(t)^m\right)-\left(S_{i+1}(t)^m-S_{i}(t)^m\right)\right)
    \end{align*}
    where we are using the notation $R_i(t)=\rmap_i[X(t)]$ and $S_i(t)=\rmap_i[Y(t)]$,
    and adopt the convention \eqref{eq:Rconvent}. % $R_{0}=S_{0}=0$.
    We now use summation by parts and get
    \begin{align*}
        & \frac{d}{dt}d_h(X(t),Y(t))^{2h}\\
        & \ = -2h\sum_{i=0}^{N-1}(x_i(t)-y_i(t))^{2h-1}\left(R_{i+1}(t)^m-S_{i+1}(t)^m\right) + 2h\sum_{i=0}^{N-1}(x_i(t)-y_i(t))^{2h-1}\left(R_{i}(t)^m-S_{i}(t)^m\right)\\
        & \ = -2h\sum_{i=0}^{N-1}\left[(x_{i-1}(t)-y_{i-1}(t))^{2h-1}-(x_{i}(t)-y_{i}(t))^{2h-1}\right]\left(R_i(t)^m-S_i(t)^m\right)\,.
    \end{align*}
    We claim that the last right-hand side above is non positive. Assume first $R_i(t)\leq S_i(t)$. This implies
    \[x_{i}(t)-x_{i-1}(t)\geq y_{i}(t)-y_{i-1}(t)\]
    which in turns implies
    \[x_{i}(t)-y_{i}(t)\geq x_{i-1}(t)-y_{i-1}(t)\]
    and therefore
    \[-\left[(x_{i-1}(t)-y_{i-1}(t))^{2h-1}-(x_{i}(t)-y_{i}(t))^{2h-1}\right]\left(R_i(t)^m-S_i(t)^m\right)\geq 0\,.\]
    Similarly, if $R_i(t)\geq S_i(t)$ then
    \[x_{i}(t)-x_{i-1}(t)\leq y_{i}(t)-y_{i-1}(t)\]
    and the claim follows. This proves the statement for $h$ finite. The estimate for $d_\infty$ follows by a standard argument in which $h\rightarrow+\infty$. We omit the details.
\end{proof}
%
% \begin{theorem}\label{thm:support}
% Let $\overline{\rho}\in \mathcal{P}(\mathbb{R})\cap L^1(\mathbb{R})$. Then, there exist nonnegative constants $A, B$ independent of $N$ such that
% \begin{equation}\label{eq:Lest}
%     \mathcal{L}[X(t)]\leq \mathcal{L}[X(0)]\left(A + Bt^{\frac{1}{m+1}}\right)\qquad \hbox{for all $t\geq 0$}\,.
% \end{equation}
% \end{theorem}
%
We are now ready to prove Theorem \ref{thm:support_main}.
\begin{proof}[Proof of Theorem \ref{thm:support_main}]
  The idea for the remaining part of the proof is to construct for a given $N\ge4$ an $N$-particle measure
  --- defined in terms of its point configuration $\bar Y^{(N)}$ and density mapping $\bar S^{(N)}$ ---
  of arbitrarily short support for which
  the corresponding expression $\bar Z$ appearing in \eqref{eq:estimate_support_1}
  has a universal lower bound $-\Lambda$, i.e.,
  \begin{align}
    \label{eq:etabound2}
    S^{(N)}_k\Delta_k\big[\big(\bar S^{(N)}\big)^m\big]\ge-\Lambda
    \quad\text{for $k=1,\ldots,N$}.
  \end{align}
  For that construction, consider the function $f:(0,1)\to\R$ given by
  \begin{align*}
    f(z) = \big(z(1-z)\big)^{-1/m},
  \end{align*}
  which --- recalling $m>1$ --- has finite integral,
  \begin{align*}
    \ell:= \int_0^1 f(z)\dd z.
  \end{align*}
  Depending on two parameters $\alpha\in\R$ and $\beta>0$ at our disposal,
  define $\bar S^{(N)}$ and $\bar Y^{(N)}$ by
  \begin{align*}
    \bar S^{(N)}_k &:= \frac1{\beta f\left(\frac{k-1/2}{N}\right)} \quad \text{for $k=1,2,\ldots,N$}, \\
    \bar y^{(N)}_j &:= \alpha + \frac1N \sum_{k=1}^j\frac1{\bar S^{(N)}_k} \quad \text{for $j=0,1,\ldots,N$};
  \end{align*}
  we further set $\bar S^{(N)}_0:=0$ and $\bar S^{(N)}_{N+1}:=0$.
  By convexity of $f$, the support length $\mathcal{L}[\bar Y^{(N)}]=\bar y^{(N)}_N-\bar y^{(N)}_0$
  can be estimated as follows:
  \begin{align}
    \label{eq:Yisshort}
    \mathcal{L}[\bar Y^{(N)}]
    = \frac1N\sum_{k=1}^N\frac1{\bar S^{(N)}_k}
    = \frac \beta N \sum_{k=1}^Nf\left(\frac{k-1/2}N\right)
    \le \beta \sum_{k=1}^N \int_{(k-1)/N}^{k/N}f(z)\dd z
    = \beta \ell.
  \end{align}
  A direct computation --- using that $f(z)^{-m}=z(1-z)/\beta^m$ --- yields that
  \begin{align*}
    \Delta_k\big[\big(\bar S^{(N)}\big)^m\big] = -\frac2{\beta^m}
    \quad \text{for $k=2,3,\ldots,N-1$},
  \end{align*}
  as well as
  \begin{align*}
    \Delta_1\big[\big(\bar S^{(N)}\big)^m\big]
    = \Delta_N\big[\big(\bar S^{(N)}\big)^m\big]
    = \frac{2N-7}{2\beta^m} >0 ,
  \end{align*}
  since we assume $N\ge4$.
  Thus
  \begin{align*}
    \min_{k=1,\ldots,N}S^{(N)}_k\Delta_k\big[\big(\bar S^{(N)}\big)^m\big]
    \ge -\frac2{\beta^m}\max_{k=2,\ldots,N-1}S^{(N)}_k
    \ge -\frac2{\beta^{m+1}}\max_{z\in(0,1)}\frac1{f(z)}
    \ge -\frac{2}{4^{1/m}\beta^{m+1}},
  \end{align*}
  which imples that \eqref{eq:etabound2} holds with
  \begin{align*}
    \Lambda =  \frac1{4^{1/m-1/2}\beta^{m+1}}.
  \end{align*}
  Now let $Y^{(N)}(t)$ be the solution to the evolution equation \eqref{eq:particles1} with initial condition $\bar Y^{(N)}$.
  Estimate \eqref{eq:estimate_support_1} on the support length
  in combination with the elementary inequality \eqref{eq:very_elementary} yields
  \begin{align*}
    \mathcal L\big[Y^{(N)}(t)\big]
    &\le \mathcal L\big[\bar Y^{(N)}\big] \big(1+(m+1)\Lambda t\big)^{1/(m+1)} \\
    &\le \beta\ell \Big(1+ \big[(m+1)\Lambda t\big]^{1/(m+1)}\Big)
    \le \beta\ell + \frac{\ell}{4^{(1/m-1/2)/(m+1)}}\big[(m+1)t\big]^{1/(m+1)}.
  \end{align*}
  In short, with $B:=4^{-(1/m-1/2)/(m+1)}(m+1)^{1/(m+1)}\ell$, we have
  \begin{align}
    \label{eq:loooong}
    \mathcal L\big[Y^{(N)}(t)\big] \le \beta\ell + Bt^{1/(m+1)}.
  \end{align}
  To finish the proof, choose $\alpha:=\frac12(\bar x^{(N)}_0+\bar x^{(N)}_N)$ as the center of $\bar X^{(N)}$'s support,
  and let $\beta>0$ be sufficiently small such that $\bar Y^{(N)}$'s support length
  is smaller than half of that of $\bar X^{(N)}$,
  i.e., $\beta\ell<\frac12\big(\bar x^{(N)}_N-\bar x^{(N)}_0\big)$.
  This choice is made such that
  \begin{align*}
    d_\infty\big(\bar X^{(N)},\bar Y^{(N)}\big)
    = \max_k \big|\bar x^{(N)}_k-\bar y^{(N)}_k\big|
    \le \max_k \big|\bar x^{(N)}_k-\alpha| + \max_k\big|\bar y^{(N)}_k-\alpha\big|
    \le \frac12\mathcal L[\bar X^{(N)}]+\beta\ell.
  \end{align*}
  By contractivity in $d_\infty$, see Proposition \ref{prop:contraction}, it follows that at any time $t>0$,
  \begin{align*}
    \big|x^{(N)}_0(t)-y^{(N)}_0(t)\big| + \big|x^{(N)}_N(t)-y^{(N)}_N(t)\big|
    &\le 2d_\infty\big(X^{(N)}(t),Y^{(N)}(t)\big) \\
    &\le 2d_\infty\big(\bar X^{(N)},\bar Y^{(N)}\big)
    \le \mathcal L[\bar X^{(N)}]+2\beta\ell.
  \end{align*}
  In combination with \eqref{eq:loooong}, we thus obtain
  \begin{align*}
    \mathcal L[X^{(N)}(t)]
    &\le \big|x^{(N)}_0(t)-y^{(N)}_0(t)\big| + \big|x^{(N)}_N(t)-y^{(N)}_N(t)\big| + \mathcal L\big[Y^{(N)}(t)\big] \\
    &\le  \mathcal L[\bar X^{(N)}]+2\beta\ell + \beta\ell + Bt^{1/(m+1)} \\
    &= \mathcal L[\bar X^{(N)}] + Bt^{1/(m+1)} + 3\beta\ell.
  \end{align*}
  The construction above works for any sufficiently small choice of $\beta>0$.
  This implies the claimed estimate \eqref{eq:Lest}.
\end{proof}

% \textcolor{red}{Marco: the same procedure may be applied to prove a uniform estimate on the support for any particle scheme for an equation of the form
% \[\rho_t = (\rho^m)_{xx}+(\rho V[\rho])_x\]
% where $V[\rho]$ is a (possibly nonlocal) velocity field which grows at most linearly in $x$ (uniformly in $\rho$).
% }

% From the above proposition and from the ODE for $R_k$ we get the inequality
% \[\dot{R}_k\geq -\frac{R_k}{(m+1)t}\,,\]
% from which we get
% \[
% \log\left[\frac{R(t)}{R_k(0)}\right]\geq \log t^{-\frac{1}{m+1}}
% \]
% and consequently
% \[R_k(t)\geq R_k(0) t^{-\frac{1}{m+1}}\,.\]
% Hence, the quantity $d_k(t)$ satisfies
% \[
% d_k(t)\leq d_k(0) t^{\frac{1}{m+1}}\,.
% \]
% Assuming the above holds for all $k=1,\ldots,N-1$ where $N$ is the total number of particles, we get
% \[
% x_N(t)-x_1(t)=h\sum_{k=1}^{N-1}d_k(t)\leq t^{\frac{1}{m+1}} h\sum_{k=1}^{N-1}d_0(t) = (x_N(0)-x_1(t))t^{\frac{1}{m+1}}\,.
% \]
% The above is a uniform estimate on the propagation of the support, which once again fits the continuum theory. Conjecture: this is enough to prove the convergence of the scheme near vacuum.

%%%%%%%%%%%%%%%%%%%%%%%%%%%%%%%%%%%%%%%%%%%%%%%%%%%%%%%
%%%%%%%%%%%%%%%%%%%%%%%%%%%%%%%%%%%%%%%%%%%%%%%%%%%%%%%
\section{Uniform $L^\infty$ decay}\label{sec:decay}
%%%%%%%%%%%%%%%%%%%%%%%%%%%%%%%%%%%%%%%%%%%%%%%%%%%%%%%
%%%%%%%%%%%%%%%%%%%%%%%%%%%%%%%%%%%%%%%%%%%%%%%%%%%%%%%
%
% The goal of this section is to obtain an $L^\infty$-estimate on discrete solutions that parallels \eqref{eq:PMELinfty}.
% With the approach that we take, we are not able to replicate the optimal constant $C_m$,
% but we derive that analogous estimate with another universal (and in particular discretization-independent) constant.
% %
% \begin{proposition}
%   \label{prp:Linfty}
%   For every solution $R$ to \eqref{eq:particles_R_2} and at every $t>0$:
%   \begin{align}
%     \label{eq:Linfty}
%     \max_k R_k(t) \le \left(\frac{m+1}{16m\,t}\right)^{1/(m+1)}.
%   \end{align}
% \end{proposition}
% %
In this section we prove the uniform $L^\infty$-decay estimate of Theorem \ref{thm:Linfty_main}.
For the reader's convenience, we first discuss the derivation of \eqref{eq:Linfty} in the continuum setting and motivate our choice of the strategy of the proof in the discrete setting.

The classical proof of the sharp form \eqref{eq:Linfty} of the decay estimate at a given time $t>0$
is to maximize the supremum of the spatial profile $x\mapsto U(x):=\rho(t;x)$ under the constraint of given mass,
subject to the estimate \eqref{eq:AB}.
In other words, due to the translation invariance one may assume without restriction that the maximum of $U$ is attained at $x=0$. Then,
one seeks to find $U(0)$
subject to
\begin{align*}
  (U^{m-1})'' = -\frac{m-1}{m(m+1)t}
  \quad \text{on its support $B$, and}\quad
  \int_B U\dd x = M.
\end{align*}
It is easily seen that an optimizing profile (actually: the unique one up to translations)
is indeed given by $U(x)=\hat\rho(t;x)$, the corresponding profile of the self-similar solution from \eqref{eq:PMEss}.
Since $\hat\rho$ is a solution, the sharpness of the estimate follows automatically.

Translating this strategy to the discrete setting is not straightforward,
since dealing with the discrete counterpart of the self-similar profile is not immediate.
Instead, we shall provide a different proof, which leads to the same decay estimate, but with a sub-optimal constant.
The advantage of this proof is that it is sufficiently robust to be directly rephrased for the Lagrangian discretization. We sketch below the proof of the continuum case for convenience.

Let $\rho$ be a compactly supported unit mass solution to \eqref{eq:PME}.
First observe that, by the estimate \eqref{eq:AB} and integration by parts,
\begin{equation}
  \label{eq:linfcalc1}
  \begin{split}
    \frac1{m(m+1)t}
    &= \frac1{m(m+1)t}\intR\rho\dd x \\
    & \ge \intR\left(-\frac{\rho^{m-1}}{m-1}\right)_{xx}\rho \dd x
    = \intR\left(\frac{\rho^{m-1}}{m-1}\right)_{x}\rho_x\dd x
    = \frac4{m^2}\intR\big(\rho^{m/2}\big)_x^2\dd x.
  \end{split}
\end{equation}
With that at hand, we find that for any $\bar x\in\R$,
\begin{equation}
  \label{eq:linfcalc2}
  \begin{split}
    \frac{\rho(\bar x)^{\frac{m+1}{2}}}{\frac{m+1}{2}}
    &= \frac12\left(\int_{-\infty}^{\bar x}\rho^{(m-1)/2}\rho_x\dd x - \int_{\bar x}^\infty \rho^{(m-1)/2}\rho_x\dd x\right) \\
    & = \frac1m\left(\int_{-\infty}^{\bar x}\rho^{1/2}\big(\rho^{m/2}\big)_x\dd x - \int_{\bar x}^{\infty}\rho^{1/2}\big(\rho^{m/2}\big)_x\dd x\right) \\
    &\le \frac1m \left(\intR \rho\dd x\right)^{1/2}\left(\intR\big(\rho^{m/2}\big)_x^2\dd x\right)^{1/2}.
  \end{split}
\end{equation}
Combining the two estimates \eqref{eq:linfcalc1} and \eqref{eq:linfcalc2}, we obtain
\begin{align*}
  \sup \rho \le \left[\frac{m+1}{2m}\left(\frac{m}{4(m+1)t}\right)^{1/2}\right]^{2/(m+1)}
  = \left(\frac{m+1}{16m\,t}\right)^{1/(m+1)}.
\end{align*}
We now adopt a discrete version of the above strategy to prove Theorem \ref{thm:Linfty_main}.
\begin{proof}[Proof of Theorem \ref{thm:Linfty_main}]
  Since by \eqref{eq:dAB2}, the upper bound
  \begin{align*}
    \frac1{(m+1)t} \ge -R_k\Delta_k(R^m)
  \end{align*}
  holds at every $k=1,2,\ldots,N$,
  the same upper bound is true also for the average of the right-hand side.
  In combination with the summation by parts rule \eqref{eq:byparts} and the elementary inequality \eqref{eq:ineq1},
  we thus obtain
  \begin{equation}
    \label{eq:Rtom}
    \begin{split}
      \frac1{(m+1)t}
      &\ge -\frac1N\sum_{k=1}^N R_k\Delta_k(R^m) \\
      &\qquad = \frac1N\sum_{k=0}^N\dff^+_kR\,\dff^+_k(R^m)
      \ge \frac{4m}{(m+1)^2}\frac 1N\sum_{k=0}^N \big(\dff^+_kR^{\frac{m+1}{2}}\big)^2.
    \end{split}
  \end{equation}
  Now, on the other hand, we have for every index $k$:
  \begin{align*}
    R_k^{\frac{m+1}{2}}
    &= \frac12 \left(
      \sum_{j=1}^{k}\big(R_j^{\frac{m+1}{2}}-R_{j-1}^{\frac{m+1}{2}}\big)+
      \sum_{j=k}^{N}\big(R_j^{\frac{m+1}{2}}-R_{j+1}^{\frac{m+1}{2}}\big)\right) \\
    &\le \frac1{2N}\sum_{j=0}^{N}\big|\dff^+_jR^{\frac{m+1}{2}}\big|
      \le \frac12\left(\frac1N\sum_{j=0}^N\big(\dff^+_jR^{\frac{m+1}{2}}\big)^2\right)^{1/2},
  \end{align*}
  where we have used again Jensen's inequality for sums in the last estimate.
  Bounding the right-hand side above further in terms using \eqref{eq:Rtom} yields the result.
\end{proof}

%%%%%%%%%%%%%%%%%%%%%%%%%%%%%%%%%%%%%%%%%%%%%%%%%%%%%%%
%%%%%%%%%%%%%%%%%%%%%%%%%%%%%%%%%%%%%%%%%%%%%%%%%%%%%%%
\section{Convergence of the scheme}\label{sec:convergence}
%%%%%%%%%%%%%%%%%%%%%%%%%%%%%%%%%%%%%%%%%%%%%%%%%%%%%%%
%%%%%%%%%%%%%%%%%%%%%%%%%%%%%%%%%%%%%%%%%%%%%%%%%%%%%%%
%

This section is devoted to the proof of Theorem \ref{thm:convergence_main}.
The proof has two main elements:
the first one is the (strong) compactness of the sequence $\left(\rho^{(N)}\right)_{N\geq 2}$ defined in \eqref{eq:approximate_density} in $L^1_{\mathrm{loc}}([0,+\infty)\times\R)$ and $L^m_{\mathrm{loc}}([0,+\infty)\times\R))$, with respective limit $\rho$;
the second one is the identification of $\rho$ as a solution to the \eqref{eq:PME} in the sense of \eqref{eq:weakPME}.

%%%%%%%%%%%%%%%%%%%%%%%%%%%%%%%%%%%%%%%%%%%%%%%%%%%%%%%
\subsection{Compactness in $L^1$ and $L^m$}
%%%%%%%%%%%%%%%%%%%%%%%%%%%%%%%%%%%%%%%%%%%%%%%%%%%%%%%
%
To establish strong compactness of $(\rho^{(N)})_{N\ge2}$
we apply the version of the Aubin-Lions theorem recalled in Proposition \ref{prp:savare} in Appendix \ref{sec:aubin}.
% \begin{proposition}[\cite[Thoerem 2]{rossi_savare}]
%     \label{prp:savare}
%     Let $T>0$. On a non-empty closed convex subset $U$ of a Banach space $B$, let be given:
%     \begin{itemize}
%         \item a lower semi-continuous functional $\fnc:U\to[0,+\infty]$ with relatively compact sublevel sets;
%         \item a lower semi-continuous map $g:U\times U\to[0,\infty]$ such that $g(u,v)=0$ implies $u=v$.
%     \end{itemize}
%     Assume that $(u_n)$ is a sequence of maps $u_n:(0,T)\to X$ from a time interval $(0,T)$ into $U$, satisfying the following:
%     \begin{enumerate}
%         \item the sequence is tight with respect to $\fnc$,
%         \begin{align}
%             \label{eq:stight}
%             \sup_n \int_0^T \fnc\big(u_n(t)\big)\dd t<\infty.
%         \end{align}
%         \item the $u_n$ are weakly equi-continuous with respect to $g$,
%         \begin{align}
%             \label{eq:sequicont}
%             \lim_{\tau\searrow0}\sup_n\int_0^{T-\tau}g\big(u_n(t+\tau),u_n(t)\big)\dd t=0.
%         \end{align}
%     \end{enumerate}
%     Then there is a subsequence $(u_{n_k})$ that converges in measure with respect to $t\in(0,T)$ in the $B$-norm to a limit $\bar u$.
% \end{proposition}
%
We are going to apply Proposition \ref{prp:savare} with the Banach space $B=L^1(\R)$ and the closed convex subset
\begin{align}
    & U = \left\{ \rho\in L^1(\R)\,\middle|\,\rho\ge0,\ \int_\R\rho\dd x=1\right\}.\label{eq:U}
\end{align}
Let us further single out the subsets $U^N\subset X$ defined by
\[
U^N=\mathrm{span}\left\{\rec^N[X]\,:\,\, X=(x_0,\ldots,x_N)\in \mathcal{K}^N\right\}\,.
\]
Our choices for $\fnc$ and $g$ in Proposition \ref{prp:savare} are the following:
\begin{itemize}
    \item $\fnc$ is the sum of the total variation functional $\tvn{\rho^{\frac{m+1}{2}}}$ and of its first moment of $\rho$, that is
    \begin{align}
    & \fnc[\rho]=\tvn{\rho^{\frac{m+1}{2}}} + M_1[\rho]\,,\label{eq:functional}\\
        & \tvn{\rho^{\frac{m+1}{2}}} = \sup\left\{\intR \partial_x\varphi\,\rho^{\frac{m+1}{2}}\dd x\,\middle|\,\varphi\in C^1(\R),\, \sup|\varphi|\le 1\right\}\,,\label{eq:tv}\\
        & M_1[\rho]=\int_\R|x|\rho(x) dx\,.\label{eq:M1}
    \end{align}
    The lower semi-continuity of $\fnc$ with respect to $L^1(\R)$-convergence follows from Lemma \ref{lem:appendix_TV} in Appendix \ref{sec:aubin} for the total variation part and from Fatou's lemma for the first moment part. Helly's Theorem (see e.g. \cite[Theorem 2.3]{bressan_book}) implies for a sequence $\rho_n$ that if $\tvn{\rho^{\frac{m+1}{2}}_n}$ is uniformly bounded then there exists a subsequence $\rho^{\frac{m+1}{2}}_{n_k}$ converging strongly in $L^1(\Omega)$ for $\Omega\subset \R$ a closed and bounded set. Fisher-Riesz theorem (see \cite[Theorem 4.9]{brezis_FA}) and Lebesgue's dominated convergence theorem then implies, for a further subsequence still denoted by $\rho_{n_k}$, almost everywhere convergence and $L^{\frac{m+1}{2}}(\Omega)$-convergence.
    H\"older's inequality then implies $\rho_{n_k}$ converges in $L^1(\Omega)$. The additional uniform bound of $M_1$ implies compactness of the sublevels of $\fnc$ in $L^1(\R)$, see Lemma \ref{lem:appendix_moment} in Appendix \ref{sec:aubin}. For the application at hand, we recall that for $\rho\in U^N$, the total variation part simplifies to
    \begin{align}
        \label{eq:tvsimple}
        \tvn{\rho^{\frac{m+1}{2}}} = \sum_{n=0}^N \left|\rho_{n+1}^{\frac{m+1}{2}}-\rho_n^{\frac{m+1}{2}}\right|
        \quad \text{with} \quad \rho=\rec^N[X]\,,\,\,X=(x_0,\ldots,x_n)\,,\,\,
        \rho_n = \frac{1/N}{x_n-x_{n-1}}.
    \end{align}
    \item $g$ is the $L^1$-Wasserstein distance $d_1$ recalled in \eqref{eq:wasserstein}, namely
    \begin{equation}\label{eq:g}
g(\rho,\eta)=d_1(\rho,\eta)=\int_0^1\left|\mathcal{T}[\rho](z)-\mathcal{T}[\eta](z)\right|dz\,,
    \end{equation}
    with $\mathcal{T}:\mathcal{P}(\R)\rightarrow \mathcal{K}$ defined in \eqref{eq:mapT1}-\eqref{eq:mapT2}.
    % also known as bounded Lipschitz distance: for $\rho, \rho'\in X$ with corresponding inverse distribution functions $R,R':[0,1]\to\R$, define
    % \begin{align*}
    %     g(\rho,\rho') = \int_0^1 |R'(z)-R(z)|\dd z.
    % \end{align*}
    For $\rho=\rec^N[X],\eta=\rec^N[Y]\in U^N$ with corresponding vectors $X=(x_0,\ldots,x_N), Y=(y_0,\ldots, y_N)\in \mathcal{K}^N$, the definition of $d_1$ simplifies since the inverse distribution functions $\mathcal{T}[\rho], \mathcal{T}[\eta]$ are piecewise linear. Specifically, $\mathcal{T}[\rho]- \mathcal{T}[\eta]$ is the affine interpolation of $y_{n-1}-x_{n-1}$ and $y_n-x_n$ over the interval $((n-1)/N,n/N)$. From there, one easily derives the following upper bound:
    \begin{align}
        \label{eq:gsimple}
        d_1(\rho,\eta) \le
        W_1^N(\rho,\eta) := \frac{1}{N}\sum_{n=0}^N |y_n-x_n|.
    \end{align}
\end{itemize}
\begin{lemma}
  \label{lem:applysavare}
  With the choices $U\subset B$ in \eqref{eq:U}, $\fnc$ in \eqref{eq:functional}-\eqref{eq:M1}, $g$ in \eqref{eq:g}, and for any $T>0$,
  the sequence $(\rho^{(N)})_{N\ge2}$ defined by
  \[\rho^{(N)}(x,t)=\rec^N[X^{(N)}(t)](x)\,,\qquad \hbox{with $X^{(N)}(t)=(x^{(N)}_0(t),\ldots,x^{(N)}_N(t))$ solving \eqref{eq:particles1},}\]
  satisfy the hypotheses \eqref{eq:stight} and \eqref{eq:sequicont} of Proposition \ref{prp:savare} in Appendix \ref{sec:aubin}.
  Consequently, there exist a subsequence $\rho^{(N_k)}$ and a limit $\rho\in L^1([0,T]\times\R)$
  with $\rho^{(N_k)}\to\rho$ in $L^1([0,T]\times\R)$.
\end{lemma}
\begin{proof}
  Fix some time horizon $T>0$. We verify the two hypotheses \eqref{eq:stight} and \eqref{eq:sequicont} for $\fnc$ and $g$ respectively.\\
  \noindent\textbf{Proof of the tightness condition \eqref{eq:stight}.}\,
  By means of the key estimate \eqref{eq:Rtom} above,
  we obtain, recalling the convention \eqref{eq:Rconvent},
  \begin{align*}
    & \tvn{(\rho^{(N)})^{\frac{m+1}{2}}}
    =\sum_{j=1}^{N+1} \big|R_j^{\frac{m+1}{2}}-R_{j-1}^{\frac{m+1}{2}}\big| = \frac{1}{N}\sum_{j=0}^{N+1}\big|\dff^+_jR^{\frac{m+1}{2}}\big]\big| \\
    &\le \left(\frac1N\sum_{j=0}^N\big(\dff^+_jR^{\frac{m+1}{2}}\big)^2\right)^{1/2}
      \le \left(\frac{m+1}{4mt}\right)^{1/2},
  \end{align*}
  Integrating in time, it follows that
  \begin{align*}
    \int_0^T\tvn{(\rho^{(N)})^{\frac{m+1}{2}}(t)}\dd t
    \le \int_0^T \left(\frac{m+1}{4m\,t}\right)^{1/2}\dd t
    \le \left(\frac{m+1}{m}\right)^{1/2}
    T^{1/2}.
  \end{align*}
  \\
\noindent\textbf{Proof of the weak equi-continuity in time \eqref{eq:sequicont}.}\,
  Fix some $s>t>0$. By the upper bound \eqref{eq:gsimple} on $g$,
  \begin{align*}
    & g\big(\rho^{(N)}(t),\rho^{(N)}(s)\big)
    \le W_1^{N}\big(\rho^{(N)}(t),\rho^{(N)}(s)\big)
    = \frac1N\sum_{j=0}^N\big|x_j(t)-x_j(s)\big|\\
    & \
    \le \int_s^t\frac1N\sum_{j=1}^N \big|\dot x_j(r)\big|\dd r
    \le \left[\int_s^t\left(\frac1N\sum_{j=1}^N \big|\dot x_j(r)\big|^2\right)\dd r\right]^{1/2}  |t-s|^{1/2}.
  \end{align*}
  Now combine the discrete Lagrangian AB-estimate \eqref{eq:dAB2} with the $L^\infty$-bound \eqref{eq:Linfty} as follows:
  \begin{align*}
   & \left(\frac{m+1}{16mt}\right)^{(m-1)/(m+1)}\cdot\frac{m-1}{m(m+1)t}
    \ge \max_k R_k^{m-1}\,\max_k\big(-R_k\Delta_k R^m\big) \\
    &\ge \frac1N\sum_{k=1}^N R_k^{m-1}\,\big(-R_k\Delta_k R^m\big) = \frac1N \sum_{k=1}^N R_k^m\big(-\Delta_k R^m\big)
      = \frac1N \sum_{j=1}^N \big[\dff^+_j(R^m)\big]^2.
  \end{align*}
  Recalling from \eqref{eq:particles1_differences} that
  \begin{align*}
    \dot x_j = -\dff^+_jR^m,
  \end{align*}
  we thus obtain that
  \begin{align*}
    W_1^{N}\big(\rho^{(N)}(t),\rho^{(N)}(s)\big)
    \le C t^{-m/(m+1)}|t-s|^{1/2}
  \end{align*}
  for some $C>0$ depending only on $m$.
This implies for every $\tau>0$:
  \begin{align*}
    \int_0^{T-\tau} g\big(\rho^{(N)}(t),\rho^{(N)}(t+\tau)\big) \dd t
    \le C \int_0^T t^{-m/(m+1)}\dd t\ \tau^{1/2}
    \le (m+1)C T^{\frac{1}{m+1}}\,\tau^{1/2}.
  \end{align*}
  From Proposition \ref{prp:savare}, we conclude convergence in measure of a subsequence $(\rho^{(N_k)})$ to a limit $\rho:[0,T]\to L^1(\R)$.
  Since the $L^1(\R)$-norm of $\rho^{(N_k)}(t)$ is $t$-uniformly bounded by one,
  dominated convergence implies immediately the convergence of $\rho^{(N_k)}$ to some limit $\rho$ in $L^1([0,T]\times\R)$.
\end{proof}
\begin{lemma}
  \label{lem:convergeLm}
  For the subsequence $(\rho^{(N_k)})$ in Lemma \ref{lem:applysavare},
  we have $\rho^{(N_k)}\to\rho$ in $L^m([0,T]\times\R)$.
\end{lemma}
\begin{proof}
 The $L^\infty$-estimate \eqref{eq:Linfty} implies
  \begin{align*}
    \big\|\rho^{(N_k)}(t)\big\|_{L^\infty(\R)}\le \left(\frac{m+1}{16m\,t}\right)^{1/(m+1)}.
  \end{align*}
  By the $L^1$-convergence of $\rho^N$ to $\rho$, the limit $\rho$ satisfies the same decay estimate.
  Therefore,
  \begin{align*}
    &\big\|\rho^{(N_k)}-\rho\big\|_{L^m([0,T]\times\R)}^m
    = \int_0^T\intR |\rho^{(N_k)}-\rho|^m\dd x\dd t \\
    &\le \int_0^T \left[\big(\|\rho^{(N_k)}(t)\|_{L^\infty(\R)}+\|\rho(t)\|_{L^\infty(\R)}\big)^{m-1}
      \intR|\rho^{(N_k)}(t)-\rho(t)|\dd x\right]\dd t \\
    &\le 2^{m-1}\left(\frac{m+1}{16m}\right)^{(m-1)/(m+1)}
      \int_0^T\left[t^{-(m-1)/(m+1)}\intR|\rho^{(N_k)}(t)-\rho(t)|\dd x\right]\dd t\\
    &\le 2^{m-1}\left(\frac{m+1}{16m}\right)^{(m-1)/(m+1)}
      \left(\int_0^Tt^{-(m-1)/m}\dd t\right)^{m/(m+1)} \\
    &\qquad \left(\int_0^T\left[\intR|\rho^{(N_k)}(t)-\rho(t))|\dd x\right]^{m+1}\dd t\right)^{1/(m+1)}\\
    &\le 2^{m-1}\left(\frac{m+1}{16m}\right)^{(m-1)/(m+1)}
      \big(mT^{1/m}\big)^{m/(m+1)} \\
    &\qquad \left(\sup_{t\in[0,T|}\big[\|\rho^{(N_k)}(t)\|_{L^1(\R)}+\|\rho\|_{L^1(\R)}\big]^m \int_0^T\intR|\rho^{(N_k)}(t)-\rho(t)|\dd x\dd t \right)^{1/(m+1)} \\
    &\le
    % 2^{m-1/(m+1)}\left(\frac{(m+1)(mT)^{1/(m-1)}}{16}\right)^{(m-1)/(m+1)}
     C_m(T) \big\|\rho^{(N_k)}-\rho_T\big\|_{L^1([0,T]\times\R)}^{1/(m+1)}\,,
  \end{align*}
  for some constant $C_m(T)$ depending on $m$ and continuously on $T\in [0,T]$. Thus the convergence in $L^1([0,T]\times\R)$ obtained in Lemma \ref{lem:applysavare} above
  induces also convergence in $L^m([0,T]\times\R)$.
\end{proof}

%%%%%%%%%%%%%%%%%%%%%%%%%%%%%%%%%%%%%%%%%%%%%%%%%%%%%%%
\subsection{Limit equation}
%%%%%%%%%%%%%%%%%%%%%%%%%%%%%%%%%%%%%%%%%%%%%%%%%%%%%%%
%
Throughout this section, let $\varphi\in C^\infty_c([0,T)\times\R)$ be fixed,
and define, for each $N\in \mathbb{N}$ the following integrals $I^{(N)},J^{(N)}, K^{(N)}\in\R$:
\begin{align*}
  I^{(N)} &:= \frac1N\sum_{j=0}^N\int_0^T\partial_t\varphi\big(t;x^{(N)}_j(t)\big)\dd t, \\
  J^{(N)} &:= \frac1N\sum_{j=1}^N\int_0^TN\big(\partial_x\varphi\big(t;x^{(N)}_j(t)\big)-\partial_x\varphi\big(t;x^{(N)}_{j-1}(t)\big)\big)\,R^{(N)}_j(t)^m\dd t, \\
  K^{(N)} &:= \frac1N\sum_{j=0}^N\varphi(0;x^{(N)}_j(0)\big).
\end{align*}
\begin{lemma}
  For every $N\ge2$,
  \begin{align}
    \label{eq:consist}
    I^{(N)}+J^{(N)}=-K^{(N)}.
  \end{align}
\end{lemma}
\begin{proof}
  By the chain rule, at every $t\in(0,T)$:
  \begin{align*}
    \frac{\dn}{\dn t}\frac1N\sum_{j=0}^N\varphi\big(t;x^{(N)}_j(t)\big)
    = \frac1N\sum_{j=0}^N\partial_t\varphi\big(t;x^{(N)}_j(t)\big)
    + \frac1N\sum_{j=0}^N\partial_x\varphi\big(t;x^{(N)}_j(t)\big)\dot x^{(N)}_j(t),
  \end{align*}
  Recall that the $x^{(N)}_j$ satisfy, with the convention \eqref{eq:Rconvent},
  \[ \dot x^{(N)}_j(t) = -\dff_j^+R^{(N)}(t)^m = -N\left(R^{(N)}_{j+1}(t)^m-R^{(N)}_j(t)^m\right),\qquad j\in \{0,\ldots,N\}\,,\]
  and so a summation by parts yields
  \begin{align*}
    \frac1N\sum_{j=0}^N\partial_x\varphi\big(t;x^{(N)}_j(t)\big)\dot x^{(N)}_j(t)
    = \sum_{j=1}^N \big(\partial_x\varphi\big(t;x^{(N)}_j(t)\big)-\partial_x\varphi\big(t;x^{(N)}_{j-1}(t)\big)\big)\,R^{(N)}_j(t)^m.
  \end{align*}
  Now integrate with respect to $t\in[0,T]$,
  and use the fundamental theorem of calculus to obtain \eqref{eq:consist}.
\end{proof}
\begin{lemma}
Let $\rho^{(N_k)}$ be the subsequence in Lemmas \eqref{lem:applysavare} and \ref{lem:convergeLm} and let $X^{(N_k)}$ be the corresponding subsequence of $X^{(N)}$. Then, under the additional assumption \eqref{eq:sampling_support_N} on the sampling map $\sam^N$, we have that
  \begin{align}
    \label{eq:converge1}
    \lim_{k\to\infty}I^{(N_k)} = \int_0^T\intR \partial_t\varphi\,\rho\dd x\dd t.
  \end{align}
\end{lemma}
\begin{proof}
  For ease of notation, we omit the explicit dependence of $x$ and $\varphi$ on $t\in(0,T)$,
  and the super-index $(N)$ on the particles positions below.
  First, let us start with the following immediate consequence of Lemma \ref{lem:applysavare}:
  \begin{align}
    \label{eq:help013}
    \lim_{k\to 0}\left|\int_0^T\intR\partial_t\varphi\,\rho^{(N_k)}\dd x\dd t
    - \int_0^T\intR\partial_t\varphi\,\rho\dd x\dd t\right|=0.
  \end{align}
  Next, rewrite $I^{(N)}$ as follows:
  \begin{align*}
    I^{(N)} &= \frac1N\sum_{j=0}^N \int_0^T\partial_t\varphi(x_j) \dd t \\
    &= \int_0^T \bigg[\frac1{2N}\big(\partial_t\varphi(x_0)+\partial_t\varphi(x_N) \\
    &\qquad +\frac12\sum_{j=0}^{N-1}(x_{j+1}-x_j)\frac{1/N}{x_{j+1}-x_j}\,\partial_t\varphi(x_j)
    +\frac12\sum_{j=1}^{N}(x_{j}-x_{j-1})\frac{1/N}{x_{j}-x_{j-1}}\,\partial_t\varphi(x_j)
      \bigg]\dd t \\
    &= \int_0^T\left[\sum_{j=1}^N(x_j-x_{j-1})R^{(N)}_j\frac{\partial_t\varphi(x_j)+\partial_t\varphi(x_{j-1})}2\right]\dd t
      + \frac1{2N}\int_0^T\big(\partial_t\varphi(x_0)+\partial_t\varphi(x_N)\big)\dd t.
  \end{align*}
  On the one hand, by global boundedness of $\partial_t\varphi$,
  \begin{align*}
    \int_0^T \bigg[\frac1{2N}\big(\partial_t\varphi(x_0)+\partial_t\varphi(x_N)\big)\bigg]\dd t
    \le \frac{\|\partial_t\varphi\|_{L^\infty}T}{2N},
  \end{align*}
  which converges to zero as $N\rightarrow+\infty$.
  On the other hand, by global boundedness of $\partial_t\partial_{xx}\varphi$, the trapezoidal rule (see e.g. \cite[Section 5.1]{atkinson}) implies that
  \begin{align*}
    \left|\int_{x_{j-1}}^{x_j}\partial_t\varphi\dd x - (x_j-x_{j-1})\frac{\partial_t\varphi(x_j)+\partial_t\varphi(x_{j-1})}2\right|
    \le \frac{\|\partial_t\partial_{xx}\varphi\|_{L^\infty}}{12}(x_j-x_{j-1})^3,
  \end{align*}
  and so, using the functional $\mathcal{L}^{N}$ introduced in \eqref{eq:functional_L}, we obtain
  \begin{align*}
    &\left|\intR \rho^{(N)}\partial_t\varphi\dd x - \sum_{j=1}^N(x_j-x_{j-1})R_j\frac{\partial_t\varphi(x_j)+\partial_t\varphi(x_{j-1})}2\right|\\
    &\le \frac{\|\partial_t\partial_{xx}\varphi\|_{L^\infty}}{12} \sum_{j=1}^N(x_j-x_{j-1})^3R_j \le \frac{\|\partial_t\partial_{xx}\varphi\|_{L^\infty}}{12N} \sum_{j=1}^N(x_j-x_{j-1})^2 \le \frac{\|\partial_t\partial_{xx}\varphi\|_{L^\infty}}{12N} \big(\mathcal{L}^N(X^{(N)}(t))\big)^2.
  \end{align*}
  Integrating in time, this implies further that
  \begin{align*}
    \left|\int_0^T\intR \rho_T^{(N)}\, \partial_t\varphi\dd x - I^{(N)}\right|
    \le \left(\frac{\|\partial_t\varphi\|_{L^\infty}T}{2}
    +\frac{\|\partial_t\partial_{xx}\varphi\|_{L^\infty}T}{12}  \big(\mathcal{L}^N(X^{(N)}(t))\big)^2\right)\frac1N.
  \end{align*}
  Using the bound \eqref{eq:Lest} of Theorem \ref{thm:support_main} on $\mathcal{L}^{N}$,
  the hypothesis \eqref{eq:sampling_support_N},
  and \eqref{eq:help013} from the first step of the proof,
  the claimed convergence \eqref{eq:converge1} now follows.
\end{proof}
\begin{lemma}
  We have that
  \begin{align}
    \label{eq:converge2}
    \lim_{k\to\infty}J^{(N_k)} = \int_0^T\intR \partial_{xx}\varphi\,\rho^m\dd x\dd t.
  \end{align}
\end{lemma}
\begin{proof}
  Again, we omit the explicit dependence of $\varphi$ and $x$ on $t\in[0,T]$,
  and the super-index $N$ of $x^{(N)}$.
  We have:
  \begin{align*}
   & J^{(N_k)}
    = \int_0^T\sum_{j=1}^N \big(\partial_x\varphi(x_j)-\partial_x\varphi(x_{j-1})\big)\,\big(R^{(N)}_j\big)^m\dd t \\
    &= \int_0^T\sum_{j=1}^N \left(\int_{x_{j-1}}^{x_j}\partial_{xx}\varphi \dd x\right)\,\big(R^{(N)}_j\big)^m\dd t= \int_0^T\intR \partial_{xx}\varphi\,\big(\rho^{(N)}\big)^m\dd x\dd t.
  \end{align*}
  Using now the convergence of $\rho^{(N_k)}$ to $\rho$ in $L^m([0,T]\times\R)$ from Lemma \ref{lem:convergeLm},
  the claim \eqref{eq:converge2} follows.
\end{proof}
\begin{lemma}
 We have
  \begin{align}
    \label{eq:converge3}
    \lim_{k\to\infty}K^{(N_k)} = \intR \varphi(0,\cdot)\overline{\rho}\dd x.
  \end{align}
\end{lemma}
\begin{proof}
  The claim follows immediately from \eqref{eq:initial_sampling} with $\psi:=\varphi(0,\cdot)\in C^\infty_c(\R)$,
  simply using that $\lim_{N\to\infty}(N+1)/N=1$.
\end{proof}
\begin{proof}[Proof of Theorem \ref{thm:convergence_main}]
  Use \eqref{eq:converge1},  \eqref{eq:converge2}, and  \eqref{eq:converge3} to pass to the limit in \eqref{eq:consist},
  which yields the very weak form \eqref{eq:weakPME}. The uniqueness result in \cite{pierre_nonlinear_analysis_1982} implies the whole sequence $\rho^{(N)}$ converges to $\rho$.
\end{proof}

\section*{Acknowledgments}
MDF is partially supported by the Italian “National Centre for HPC, Big Data and Quantum Computing” - Spoke 5 “Environment and Natural Disasters” and by the Ministry of University and Research (MIUR) of Italy under the grant PRIN 2020- Project N. 20204NT8W4, Nonlinear Evolutions PDEs, fluid
dynamics and transport equations: theoretical foundations and applications.
MDF is also partially supported by the InterMaths Network, \url{www.intermaths.eu}. MDF acknowledges the (usual) warm hospitality from his colleagues at TU Munich during his visit in May 2025, which served to give the decisive boost to this work.
DM's research has been supported by the DFG Collaborative Research Center TRR 109, ``Discretization in Geometry and Dynamics.''

\begin{appendix}

\section{A technical lemma}\label{sec:appendix_lemma}
\begin{proof}[Proof of Lemma \ref{lem:technical1}]
    We compute
    \begin{align*}
        & \sum_{k=0}^{N+1} \Delta_k F = N\sum_{k=1}^{N}\left((\dff^+ F)_k-(\dff^+ F)_{k-1}\right) +N^2 F_1+N^2 F_{N}\\
        & \ = -N(\dff^+ F)_0+ N(\dff^+ F)_N+N^2 F_1+N^2 F_N = 0.
    \end{align*}
    Hence, either $\Delta_k F=0$ for all $k=0,\ldots,N+1$ or there exists at least one index $k=0,\ldots,N+1$ such that $\Delta_k F<0$. The former case implies in particular $\Delta_{0}F=N^2F_1=0$, which in turns implies $F_2=0$ from $\Delta_1 F=N^2(F_2-2F_1)=0$, and inductively one gets $F_k=0$ for all $k=1,\ldots,N$ which proves the assertion. In the latter case, the index $k$ at which $\Delta_k F<0$ can be neither $k=0$ nor $k=N+1$, as in both cases one gets a nonnegative quantity.
\end{proof}

\section{Some compactness results}\label{sec:aubin}
We start with a version of Aubin-Lions lemma.
\begin{proposition}[Theorem 2 of \cite{rossi_savare}]
    \label{prp:savare}
    Let $T>0$. On a non-empty closed convex subset $U$ of a Banach space $B$, let be given:
    \begin{itemize}
        \item a lower semi-continuous functional $\fnc:U\to[0,+\infty]$ with relatively compact sublevel sets;
        \item a lower semi-continuous map $g:U\times U\to[0,\infty]$ such that $g(u,v)=0$ implies $u=v$.
    \end{itemize}
    Assume that $(u_n)$ is a sequence of maps $u_n:(0,T)\to X$ from a time interval $(0,T)$ into $U$, satisfying the following:
    \begin{enumerate}
        \item the sequence is tight with respect to $\fnc$,
        \begin{align}
            \label{eq:stight}
            \sup_n \int_0^T \fnc\big(u_n(t)\big)\dd t<\infty.
        \end{align}
        \item the $u_n$ are weakly equi-continuous with respect to $g$,
        \begin{align}
            \label{eq:sequicont}
            \lim_{\tau\searrow0}\sup_n\int_0^{T-\tau}g\big(u_n(t+\tau),u_n(t)\big)\dd t=0.
        \end{align}
    \end{enumerate}
    Then there is a subsequence $(u_{n_k})$ that converges in measure with respect to $t\in(0,T)$ in the $B$-norm to a limit $\bar u$.
\end{proposition}

Next we prove some technical lemmas.
\begin{lemma}\label{lem:appendix_TV}
    Let $\rho_n\rightarrow \rho$ in $L^1(\R)$. Then, for all $\gamma>1$ there holds
    \begin{equation}\label{eq:lower_semi}
        \tvn{\rho^{\gamma}}\leq \liminf_{n\rightarrow+\infty}\tvn{\rho^\gamma_n}
    \end{equation}
\end{lemma}

\begin{proof}
Let $\rho_{n_k}$ be a subsequence of $\rho_n$ such that
\[\lim_{k\rightarrow+\infty}\tvn{\rho^\gamma_{n_k}}=\liminf_{n\rightarrow+\infty}\tvn{\rho^\gamma_n}\,.\]
Without restriction, due to the convergence of $\rho_{n_k}$ to $\rho$ in $L^1(\R)$ and to the classical Fisher-Riesz theorem (see \cite[Theorem 4.9]{brezis_FA}), we may assume that $\rho_{n_k}\rightarrow \rho$ almost everywhere on $\R$. Helly's Theorem (see e.g. \cite[Theorem 2.3]{bressan_book}) implies that, up to extracting a further subsequence we keep denoting by $\rho_{n_k}$ for simplicity,
\begin{equation}\label{eq:almost_ev}
    \rho_{n_k}^\gamma\rightarrow \eta\qquad \hbox{almost everywhere on $\R$ and in $L^1_{\mathrm{loc}}(\R)$}\,,
\end{equation}
where have also used once again Fisher-Riesz theorem and a standard diagonal procedure. The above implies $\rho_{n_k}\rightarrow \overline{\rho}:=\eta^{1/\gamma}$ almost everywhere. Hence, $\overline{\rho}=\rho$ almost everywhere and $\rho_{n_k}^\gamma \rightarrow \rho^\gamma$ almost everywhere and in $L^1_{\mathrm{loc}}(\R)$. We then apply the standard $L^1_{\mathrm{loc}}$ lower semi-continuity of the total variation, see \cite[Theorem 1, Subsection 5.2.1]{evans_gariepy} to obtain
\[\tvn{\rho^\gamma}\leq \lim_{n\rightarrow+\infty}\tvn{\rho^\gamma_{n_k}}=\liminf_{n\rightarrow+\infty}\tvn{\rho^\gamma_n}\]
\end{proof}

\begin{lemma}\label{lem:appendix_moment}
  Assume the sequence $(\rho^n)_n$ is compact in $L^1(\Omega)$ for any $\Omega\subset \R$ compact and satisfies $\rho_n\geq 0$ $\mathcal{L}^1$-almost everywhere. Assume $\int_\R |x|\rho^n(x)dx$ is uniformly bounded. Then $\rho^n$ is compact in $L^1(\R)$.
\end{lemma}

\begin{proof}
By choosing an increasing sequence of intervals $\Omega_k=[-k,k]\subset \R$, we construct a subsequence of $\rho^n$ which converges $\mathcal{L}^1$-almost everywhere and strongly in $L^1(\Omega_k)$ for all $k\in \mathbb{N}$ to some locally integrable function $\rho\geq 0$. By Fatou's lemma we infer that $\rho\in L^1(\R)$ and that
\[\int_\R|x|\rho(x) dx \leq\liminf_{n \rightarrow +\infty}\int_\R |x|\rho^n(x) dx<+\infty\,.\]
Then, for a given $k>0$, we denote by $\rho^n$ the above subsequence and compute
    \begin{align*}
        & \int_\R \left|\rho^n(x)-\rho(x)\right| dx = \int_{-k}^k\left|\rho^n(x)-\rho(x)\right|dx +\int_{\R\setminus[-k,k]} \left|\rho^n(x)-\rho(x)\right| dx\\
        & \ = \int_{-k}^k\left|\rho^n(x)-\rho(x)\right| dx + \int_{\R\setminus[-k,k]} \frac{1}{|x|}|x|\left|\rho^n(x)-\rho(x)\right|dx\\
        & \ \leq \int_{-k}^k\left|\rho^n(x)-\rho(x)\right|dx +\frac{1}{k}\int_\R |x|\left|(\rho^n(x)+\rho(x)\right)dx\,.
    \end{align*}
    Let $\varepsilon>0$, and let $R=C/\varepsilon$ where $C>0$ is an upper bound for
    \[\left\{\int_{\R}|x|\rho(x) dx\right\}\cup \left\{\int_\R |x|\rho^n(x)dx\,:\,\,n\in \mathbb{N}\right\}.\] We get
\begin{align*}
    & \limsup_{n\rightarrow+\infty} \int_\R \left|\rho^n(x)-\rho(x)\right| dx \leq \limsup_{n\rightarrow+\infty}\int_{-[R]-1}^{[R]+1}\left|\rho^n(x)-\rho(x)\right| dx +\varepsilon= \varepsilon
\end{align*}
and the statement follows from the arbitrariness of $\varepsilon$.
\end{proof}

  \section{Elementary inequalities}\label{sec:appendix_inequalities}
\begin{lemma}
    For any $a>0$ and $m>0$, there holds
    \begin{equation}\label{eq:very_elementary}
        (1+a)^{\frac{1}{m+1}}\leq 1+a^{\frac{1}{m+1}}
    \end{equation}
\end{lemma}
\begin{proof}
 Given $f:[0,+\infty)\rightarrow \R$
 \[f(a)=1+a^{\frac{1}{m+1}}-(1+a)^{\frac{1}{m+1}}\,,\]
 the statement follows from $f(0)=0$ and
 \[f'(a)=\frac{1}{m+1}a^{-\frac{m}{m+1}}\left(1-\left(\frac{a}{1+a}\right)^{\frac{m}{m+1}}\right)\geq 0\,.\]
\end{proof}

  \begin{lemma}
    For any $p>0$, the following inequality holds for all $y,z\ge0$:
    \begin{align}
      \label{eq:ineq1}
      (z-y)(z^p-y^p) \ge \frac{4p}{(p+1)^2}\big(z^{(p+1)/2}-y^{(p+1)/2}\big)^2.
    \end{align}
  \end{lemma}
  \begin{proof}
    Without loss of generality, assume that $z>y\ge0$.
    Divide \eqref{eq:ineq1} by $p(z-y)^2$ to obtain the following equivalent form:
    \begin{align*}
      \frac{z^p-y^p}{p(z-y)} \ge \left(\frac{z^{(p+1)/2}-y^{(p+1)/2}}{\frac{p+1}2\,(z-y)}\right)^2.
    \end{align*}
    Now rewrite both sides as averaged integrals:
    \begin{align*}
      \fint_y^z \zeta^{p-1}\dd\zeta \ge \left(\fint_y^z\zeta^{(p-1)/2}\dd\zeta\right)^2.
    \end{align*}
    This is a direct consequence of Jensen's inequality.
  \end{proof}
  %
  % \begin{lemma}
  %   For any $p>1$ the following inequality holds for all $y,z\ge0$:
  %   \begin{align}
  %     \label{eq:summ}
  %     |z-y|^p \le |z^p-y^p|.
  %   \end{align}
  % \end{lemma}
  % %
  % \begin{proof}
  %   Without loss of generality, assume that $z>y\ge0$.
  %   Divide \eqref{eq:summ} by $z$ and rewrite it in terms of $u:=(z-y)/z$:
  %   \begin{align*}
  %     1 + u^p \le (1+u)^p.
  %   \end{align*}
  %   This is Bernoulli's inequality.
  % \end{proof}

\end{appendix}

%
%\bibliographystyle{plain}
%\bibliography{main}

\end{document}